\documentclass[10pt, a4paper]{amsart}
\usepackage[british]{babel}
\usepackage{tikz}
\usepackage{mathrsfs}
\usetikzlibrary{arrows}
\usepackage{amsmath}
\usepackage{amssymb}
\usepackage{amsthm}
\usepackage{graphicx}
\usepackage{url}
\usepackage{enumerate}
\usepackage{xcolor}
\usepackage{mathrsfs} 
\usepackage{t1enc}
\usepackage{mathtools}
\usepackage{placeins} 
\usepackage{tikz}
\usepackage{textcomp}
\usepackage{soul}
\usetikzlibrary{matrix}
\usetikzlibrary{arrows,automata}
\usetikzlibrary{decorations.pathreplacing}

\newtheorem{thm}{Theorem}[section]
\newtheorem{lem}[thm]{Lemma}

\newtheorem{cor}[thm]{Corollary}

\newtheorem{defn}[thm]{Definition}
\newtheorem{rem}[thm]{Remark}
\newtheorem{exmp}[thm]{Example}

\DeclareMathOperator{\Int}{int}

\DeclareMathOperator{\orb}{Orb}

\DeclareMathOperator{\Orb}{Orb}

\newcommand{\eps}{\varepsilon}

\newcommand{\N}{\mathbb{N}}

\newcommand{\f}{f\colon [0,1]\to[0,2]}

\numberwithin{equation}{section}

\begin{document}

\title[Renormalization in Lorenz maps]{Renormalization in Lorenz maps - completely invariant sets and periodic orbits}

\author{\L{}ukasz Cholewa}
\address[\L{}. Cholewa]{
Krakow University of Economics, Department of Statistics, Rakowicka 27, 31-510 Krak\'ow , Poland
}
\email[\L{}. Cholewa]{lcholewa94@gmail.com}

\author{Piotr Oprocha}
\address[P.~Oprocha]{AGH University of Krakow, Faculty of Applied
	Mathematics, al.
	Mickiewicza 30, 30-059 Krak\'ow, Poland
	-- and --
	Centre of Excellence IT4Innovations - Institute for Research and Applications of Fuzzy Modeling, University of Ostrava, 30. dubna 22, 701 03 Ostrava 1, Czech Republic.}
\email[P.~Oprocha]{oprocha@agh.edu.pl}


\begin{abstract}
The paper deals with dynamics of expanding Lorenz maps, which appear in a natural way as Poincar\`e maps in geometric models of well-known Lorenz attractor. Using both analytical and symbolic approaches, we study connections between periodic points, completely invariant sets and renormalizations. 
We show that some renormalizations may be connected with completely invariant sets while some others don't. We provide an algorithm to detect the renormalizations that can be recovered from completely invariant sets. Furthermore, we discuss the importance of distinguish one-side and double-side preimage. This way we provide a better insight into the structure of renormalizations in Lorenz maps.
These relations remained unnoticed in the literature, therefore we are correcting some gaps existing in the literature, improving and completing to some extent the description of possible dynamics in this important field of study.

\end{abstract}

\maketitle

\section{Introduction}

The main object of our study is \textit{expanding Lorenz maps}, that is 
maps $f\colon [0,1]\to [0,1]$ satisfying the following three conditions:
\begin{enumerate}
\item there is a \textit{critical point} $c\in (0,1)$ such that $f$ is 
continuous and strictly increasing on $[0,c)$ and $(c,1]$;
\item $\lim_{x\to c^{-}}f(x)=1$ and $\lim_{x\to c^{+}}f(x)=0$;
\item $f$ is differentiable for all points not belonging to a finite 
set $F\subseteq [0,1]$ and $\inf_{x\not\in F} f'(x)>1$;
\end{enumerate}
with special emphasis on piecewise linear case. It turned out that such maps appear in a natural way
as Poincar\`e maps in geometric models of well known Lorenz attractor. This description was independently appeared in works of Guckenheimer \cite{Gu1},
Williams \cite{Wi1} and Afraimovich, Bykov and Shil`nikov \cite{ABS1} as a tool for better understanding of chaotic behavior present in Lorenz model (cf. more recent \cite{SL}).
In 1985, Hofbauer performed in \cite{Hof1} an analysis of possible periods of periodic points in Lorenz maps
 by introducing special type of infinite graphs, induced by evolution of pieces of monotonicity under iteration of Lorenz map (see \cite{Hof2} for further extensions of this approach).
While a characterization similar to Scharkovsky was provided later in \cite{ALMT} by different methods, the tools developed by Hofbauer are of independent interest, since
they provide a deep insight into dynamics of Lorenz maps. An extensive study on utility of Hofbauer's approach, kneading theory of Milnor and Thurston and many other interesting results on dynamics of Lorenz maps can be found in PhD thesis of M. St. Pierre, see \cite{MSP2} (cf. \cite{MSP1}). Even the simplest case of maps $x\mapsto \beta x$ generates many difficulties and leads to deep results. It is usually performed at the level of symbolic dynamics, so-called $\beta$-shifts. In \cite{Schme},  Schmeling showed that many of these maps do not have the specification property, and later Thompson \cite{Tho} provided nice representations of these systems together with their further descriptions (see \cite{CT} for recent advances that originated from this, together with a nice survey on studies on structure of entropy and topological pressure). Mixing in $x\mapsto \beta x+\alpha$
also reveals complicated structure of parameters  \cite{OPR} (cf. \cite{KHK} for measure-theoretic mixing in a similar family of maps). 

Our primary model of study will be Lorenz maps with constant slope and their renormalizations. 
Since renormalization is again a Lorenz map, maps with constant slope can have only finitely many renormalizations. In the literature there are known examples of Lorenz maps, where renormalization process does not end, which has many interesting consequences. A nice example of this type is PhD thesis of Winckler \cite{Winc}.
In fact, good understanding of renormalizations can be beneficial in analysis of dynamics as it is a kind of ``zooming'' revealing periodic structure of the map, see for example \cite{Mar} for unimodal maps. However it is also worth emphasizing, that there is a huge difference between finitely and infinitely renormalizable maps, with the later having less regular dynamical behavior (see \cite{MW} or \cite{MW94}). From the above discussion we see that Lorenz maps with constant slope do not cover the whole variety of possibilities. On the other hand, many Lorenz maps can be reduced to constant slope case by proper rescaling of the domain (cf. \cite{Glen}).

In her PhD thesis \cite{Palmer} Palmer studied connections between renormalizations and special types of periodic points. This approach was extended by Glendening \cite{Glen},
Glendingning and Sparrow \cite{GS} and various other mathematicians.
In this paper we will be mainly focus on approach from \cite{Cui,Ding}, which aim to characterize renormalizations and related properties in terms of completely invariant sets.
One of the main results in \cite{Ding} states the following.
\begin{thm}[{\cite[Thm~A]{Ding}}]\label{thm:A}
	Let $f$ be an expanding Lorenz map. 
	\begin{enumerate}
		\item\label{thm:A:1}
Suppose $E$ is a proper completely invariant closed set of $f$, put
$$
e_- =\sup\{x\in E, x<c\}, 
\qquad e_+ =\inf\{x\in E, x>c\}
$$
and
$$
l = N((e_-, c)),\qquad r = N((c,e_+))
$$
Then
$$
f^l(e_-)=e_-,\qquad f^r(e_+)=e_+
$$
and the following map
\begin{equation}
R_Ef(x)=\begin{cases}
f^l(x),& x\in [f^r(c_+),c)\\
f^r(x),& x\in (c,f^l(c_-)]\\
\end{cases}
\label{eq:RE}
\end{equation}
is a renormalization of $f$ (in particular is a well defined self map of  $[f^r(c_+),f^l(c_-)]$).
 
 \item\label{thm:A:2} If $g$ is a renormalization of $f$, then there exists a unique
proper completely invariant closed set $B$ such that $R_Bf = g$.
\end{enumerate}
\end{thm}
Theorem~\ref{thm:A} is aiming at providing a beautiful connection between renormalizations and completely invariant sets.
We are going to show that unfortunately, this result is not true as stated, however we are able to provide some possible alternatives (cf. Lemma~\ref{lem:inv} or Example~\ref{ex:3}).
This way we show that some connections exist, however their complete description in full generality of Lorenz maps may be a challenging problem.
At this point it is worth to mention closely related recent studies of R. Alcazar Barrera  \cite{RB} (see also \cite{BBK}), initiated in his PhD thesis, on invariant sets arising from omitting gaps
is subshifts. Related interesting results were also obtained by Clark in her PhD \cite{Clark} (see also \cite{Clark2}).

Before we present another result from the literature, let us recall one important definition.
 Suppose $g = (f^m, f^k)$ is a renormalization map of an expanding Lorenz map $f$. 
We say that $g$ is \textit{minimal}, if any other renormalization $\hat{g} = (f^s, f^t)$ of $f$ satisfies $s\geq m$, $t\geq k$.
Let $\kappa$ be the minimal period over all periodic points of an expanding Lorenz map $f$. Then we say that $\kappa$ is the minimal period of $f$.
Another main result of \cite{Ding} tries to relate renormalizations with periodic structure in Lorenz map. 
Strictly speaking, it asserts the following:
\begin{thm}[{\cite[Thm~B]{Ding}}]\label{thm:B}
Let $f$ be an expanding Lorenz map with minimal period $\kappa$, $1 < \kappa < \infty$. Then we have the following statements:
\begin{enumerate}
\item\label{thm:B:1} $f$ admits a unique $\kappa$-periodic orbit $O$.
\item\label{thm:B:2} $D = \overline{\bigcup_{n=0}^\infty f^{-n}(O)}$ is the unique minimal completely invariant closed set of $f$.
\item\label{thm:B:3} $f$ is renormalizable if and only if $[0,1]\setminus D\neq \emptyset$. If $f$ is renormalizable, then $R_Df$, the
renormalization associated to $D$ by \eqref{eq:RE}, is the unique minimal renormalization of $f$.
\item\label{thm:B:4} The following trichotomy holds: (i) $D = [0,1]$, (ii) $D =O$, (iii) $D$ is a Cantor set.
\end{enumerate}
\end{thm}
As we will see later, in particular in Remark~\ref{rem:thmb},
Example~\ref{51from4} and Example~\ref{ex:5_2_cycle}, some arguments in the proof of Theorem~\ref{thm:B} fail, leading to incorrect statements. However, still some statements are valid.

Let us also recall an important definition from \cite{Ding}, motivated by Theorem~\ref{thm:B}. We adjust it to our setting, as otherwise it may be rarely satisfied (see Example~\ref{ex:5_2_cycle}).
A renormalization $g$ is said to be \textit{periodic} if there exists a periodic orbit $O$ such that the set $D = \overline{\bigcup_{n=0}^\infty f^{-n}(O)}$ satisfies $D=O$ and $g=R_D f$. We say that $f$ is \textit{periodically renormalizable}, provided that renormalization $R_D f$ of $f$ associated to minimal cycle $O$ of $f$ is well defined and periodic.
Mentioned above theorems from \cite{Ding} were an inspiration to some results in \cite{Cui}.

To state these results, let us
consider a family of piecewise linear maps of the form:

$$
f:=f_{a,b,c}(x)=
\begin{cases}
	
	ax+1-ac, & x\in [0,c)\\
	b(x-c), & x\in (c,1]\\
	
\end{cases} 
$$

and we additionally assume that they are Lorenz maps, that is $0\leq f(0)<f(1)\leq 1$ and $a,b>0$. 
Note that $f(0)<f(1)$
is equivalent to $1<ac+b(1-c)$. Following the authors of \cite{Cui} we denote by $\mathcal{L}$
the class of maps $f_{a,b,c}$ satisfying $0\leq f(0)<f(1)\leq 1$ and with different slopes, i.e. $a\neq b$.
Now we are able to state the following result from \cite{Cui}.

\begin{thm}[Main Theorem in \cite{Cui}, condition (1)]\label{thm:CuiMain}
Let $f \in \mathcal{L}$. Then exactly one of the following conditions holds:
\begin{enumerate}
	\item $f$ is periodic renormalizable and rotation interval is degenerated,
	\item $f$ is not renormalizable and rotation interval is non-degenerated.
\end{enumerate}
\end{thm}

To avoid confusion, let us emphasize at this point an important notation. In his paper \cite{Ding}, Ding uses the following definition of the preimage of a point $x\in[0,1]$: for any number $k\in\N$, we say that the point $y$ is a (double-side) $k$-preimage of $x$, if $\lim_{z\to y_+}f^k(z)=\lim_{z\to y_-}f^k(z)=x$. We denote the set of all $k$-preimages of $x$ in this sense by $\{f^{-k}(x)\}$ to distinguish it from the set of all standard (one-side) $k$-preimages $f^{-k}(x)$. Let us also note that $\{f^{-1}(0)\}=\{f^{-1}(1)\}=\emptyset$, which suggests that Ding uses one-side preimage when defining completely invariance (cf. \cite[Lemma~3.1]{Ding}). Unfortunately, the difference between these two notions of a preimage is not sufficiently addressed in \cite{Cui} and \cite{Ding}, which may lead to inaccuracies. The example below illustrates the importance of distinguish one-side and double-side preimage.

\begin{exmp}\label{ex:OandD}
Denote $f:=f_{a,b,c}$ for
	$$
	a=1+\sqrt{7},\quad b=\frac{1}{3}(\sqrt{7}-1),\quad\text{ and }\quad c=\frac{1}{4}.
	$$
Then $0<f(0)<f(1)<1$ and $a\neq b$, that is $f\in\mathcal{L}$.
Observe that $f^2(0)=f(1)$, so $O:=\{f(0),f(1)\}$ is minimal $2$-cycle. Clearly 
$\bigcup_{n=0}^\infty f^{-n}(O)$ is dense in $[0,1]$ as it contains $c$.
In particular it means that $f$ is renormalizable but not periodic renormalizable according to definition in \cite{Cui} (i.e. $D\neq O$), however the rotation interval 
is degenerated because $f^2[f(0),f(1)]\subset [f(0),f(1)]$ which immediately implies that the rotation number of all points is $1/2$ (see argument at p. 249 in \cite{Cui}).
\end{exmp}
The above example has also other consequences. Namely, statements of Theorem~\ref{thm:B} assert that $D$ should lead to a minimal renormalization.
This fact is heavily used in the proof of \cite[Thm~C]{Ding}. However, Example~\ref{ex:OandD} shows that while a renormalization exists, it cannot be derived from the set $D$.
Note that while $b<1$ in $f:=f_{a,b,c}$ in Example~\ref{ex:OandD}, it is not hard to check that it is conjugate to an expanding Lorenz map $T_f$ from Example 6.1 in \cite{OPR}.
Let us also stress once again the fact, that even when $D$ leads to renormalization, it does not have to be minimal, see Example~\ref{ex:5_2_cycle}.

Finally, in \cite{Bru1}, the authors raise the question about parameters $\beta $ for which generalized $\beta$-transformation has matching (this class contain all linear Lorenz maps). They ask whether there exist  generalized $\beta$-transformations with non-Pisot and non-Salem number $\beta$ and matching (see the end of p.72 in \cite{Bru1}). Our Example~\ref{51from4} answers this question affirmatively, showing that such map exists. 

The paper is organized as follows. In Section~\ref{sec:pre} we recall preliminary facts and basic properties related to Lorenz maps used later.
In Section~\ref{sec:ThmA} we present several examples related to Theorem~\ref{thm:A} and characterize some completely invariant sets in case of maps with primary $n(k)$-cycles.
We also provide partial characterization of cases, when renormalizations can be characterized by these sets, and also highlight cases when renormalization exists but does not define proper completely invariant subset.
Section~\ref{sec:thmB} studies minimal periodic orbits in the spirit of Theorem~\ref{thm:B}, in particular their connections to completely invariant sets and renormalizations.
 We also prove that transitive expanding Lorenz maps are strongly transitive.
Section~\ref{sec:periodic} considers periodic renormalizations and the last section provides an algorithm to detect the renormalizations that can be successfully recovered from the corresponding sets $J_g$.

\section{Preliminaries}\label{sec:pre}
In what follows, $\N$ always denotes the set of positive integers {and $\mathbb{N}_0=\mathbb{N}\cup\{0\}$}; $B(x,r)$ denotes open ball with center $x$ and radius $r$; \textit{Cantor set} is any compact, totally disconnected set without isolated points.

\subsection{Lorenz maps}

As we see from the definition of Lorenz map $f$, it has discontinuities. It is also not completely clear if we should define $f(c)=0$ or $f(c)=1$,
since both assignments are equally reasonable.
To deal with this problem, we will use standard doubling points construction 
(see e.\,g.\ \cite{Rai2} for details) changing phase space $[0,1]$ into Cantor set $\mathbb{X}$, and $f$ into continuous map $\hat{f}$.
In this 
construction all elements in $\left(\bigcup_{n=0}^{\infty}
f^{-n}\{c\}\right)\setminus\{0,1\}$ are doubled 
(we perform a kind of Denjoy 
extension). We easily see that this new space differs from the 
original interval~$[0,1]$ by at most countably many points and we do not modify endpoints, hence clearly the new space $\mathbb{X}$
is a Cantor set. 
If $I_e$ is inserted ''hole'' in place of a point $e$, and since $f$ preserves orientation of $[0,1]$,
when defining $\hat{f}\colon \mathbb{X}\to \mathbb{X}$, we send its endpoints into respective endpoints of the hole $I_{f(e)}$,
if there is a hole related to $f(e)$ or into the point $f(e)$ if it was not blown up (i.e. when it is $0$ or $1$).
Also, if we denote by $I_c=[c_-,c_+]$ the hole related to critical point, and $f(0)=c$ then we define $\hat{f}(0)=c_+$, and the case $f(1)=c$
is dealt analogously. It is done the same way, when $f^n(0)=c$ but this time  $\hat{f}(0)=a_+$ where $I_{f(0)}=[a_-,a_+]$. Note that in this case
also $a_-$ will be in image of $\hat{f}$, because $f$ is expanding and so in particular $f(1)>f(0)$. We have also a natural continuous projection $\pi\colon \mathbb{X}\to [0,1]$
``gluing'' inserted points back together.

Following \cite{Rai2}, we define the linear order on $\mathbb{X}$ as follows. Let $x$, $y\in\mathbb{X}$. If $x, y$ are the endpoints of the same hole $I_a=[a_-,a_+]$, then
	$$
	x<y\iff x=a_-\hspace{0.2cm}\text{and}\hspace{0.2cm}y=a_+.
	$$
Otherwise
	$$
	x<y\iff\pi(x)<\pi(y).
	$$
For $x$, $y\in\mathbb{X}$ with $x<y$ we denote
$$
n(x,y):=\min\left\lbrace k\in\mathbb{N}_0\hspace{0.2cm}|\hspace{0.2cm}\exists z\in\left(\bigcup_{j=0}^k f^{-j}(c) \right)\setminus\{0,1\}:x<z_+\hspace{0.2cm}\text{and}\hspace{0.2cm}z_-<y \right\rbrace.
$$
Then we can define the metric on $\mathbb{X}$ as
\begin{equation}
\label{eq:metricX}
d(x,y):=
\begin{cases}
|\pi(x)-\pi(y)| +\frac{1}{N(x,y)+1},& x\neq y\\
0, & x=y\\
\end{cases},
\end{equation}
where
$$
N(x,y):=
\begin{cases}
	n(x,y); & x<y\\
	n(y,x); & x>y\\
\end{cases}.
$$
Note that the topology generated by the metric $d$ coincides with the order topology on $\mathbb{X}$.
\begin{rem}\label{rem:density}
By \eqref{eq:metricX}, clearly the density of the set $\bigcup_{n\geq 0}f^{-n}(c)$ in $[0,1]$ implies the density of the sets $\bigcup_{n\geq 0}\hat{f}^{-n}(c_-)$ and $\bigcup_{n\geq 0}\hat{f}^{-n}(c_+)$ in $\mathbb{X}$.
\end{rem} 

If $u\in [0,1]$ and $|\pi^{-1}(u)|=1$ then we simply write $u$
to denote the unique point in $\pi^{-1}(u)$. Such identification should cause no confusion. Let us note that by notation $\hat{f}(a_+)$ we mean the image of $a_+\in\mathbb{X}$ by $\hat{f}$ and by the notation $f(a_+)$ we mean the right-sided limit of $f$ at a point $a\in[0,1]$. Similarly for $a_-$.

\begin{rem}
While we assume in the definition of Lorenz map that
$f$ is differentiable for all points not belonging to a finite 
set $F\subseteq [0,1]$ and $\inf_{x\not\in F} f'(x)>1$, in fact in many cases we only use the property that 
the set $\bigcup_{n\geq 0}f^{-n}(c)$ is dense in $[0,1]$.
\end{rem}

\subsection{Topological dynamics}
We say that a continuous map $T\colon X\to X$ acting on a compact metric 
space is \textit{(topologically) transitive} if for every two nonempty 
open sets $U,V\subseteq X$ there is an integer $n>0$ such that 
$T^n(U) \cap V\neq \emptyset$. The map $T$ is 
\textit{(topologically) mixing} if for every two nonempty open sets 
$U,V\subseteq X$ there is an $N>0$ such that for every $n>N$ we have 
$T^n(U) \cap V\neq \emptyset$.
If $f$ is a Lorenz map, then we say that $f$ is transitive or mixing, if induced map $\hat f$ on $\mathbb{X}$ is transitive or mixing, respectively. Note that if $U,V\subset \mathbb{X}$ are open, then $\hat{f}^n(U)\cap V\neq \emptyset$ iff $\Int f^n(\pi(U))\cap \pi(V)\neq \emptyset$. Clearly, any open set $U\neq\emptyset$ contains open and nonempty $U'\subset U$ such that  $U'=\pi^{-1}(\pi(U'))$.

The following definition can be stated for continuous maps on compact metric spaces in a slightly different way. We will use
approach from \cite{GS} (cf. \cite{OPR}) in the setting of Lorenz maps.
Suppose that $f$ is an expanding Lorenz map. Then 
$f$ is \textit{locally eventually onto} if 
for every nonempty open subset~$U\subseteq [0,1]$ there 
exist open intervals $J_1,J_2\subseteq U$ and 
$n_1,n_2\in\mathbb{N}$ such that $f^{n_1}$ maps $J_1$ 
homeomorphically to $(0,c)$ and $f^{n_2}$ maps $J_2$ 
homeomorphically to $(c,1)$.

The following is an extension of the above definition, introduced in \cite{OPR} to describe lack  of renormalization (it turned out that locally eventually onto is not enough).
	An expanding Lorenz map $f$ is 
	\textit{strongly locally eventually onto} if for every 
	nonempty open subset~$U\subseteq [0,1]$ there exist open 
	intervals $J_1,J_2\subseteq U$ and 
	$n_1,n_2\in\mathbb{N}$ such that:
	\begin{enumerate}
		\item $f^{n_1}$ maps $J_1$ homeomorphically to $(0,c)$,
		\item the restriction of~$f^k$ to~$J_1$ is continuous 
		for all~$k\in\{0,1,\dots,n_1\}$,
		\item $f^{n_2}$ maps $J_2$ homeomorphically to $(c,1)$,
		\item the restriction of~$f^k$ to~$J_2$ is continuous 
		for all~$k\in\{0,1,\dots,n_2\}$.
	\end{enumerate}

\begin{rem}
Note that if we pass from $[0,1]$ to $\mathbb{X}$ then definition of locally eventually onto is more than enough to obtain standard definition of locally eventually onto for maps acting on compact metric spaces.
\end{rem}

\subsection{Lorenz maps and renormalization}
Let $f$ be an expanding Lorenz map and let $\hat{f}$ be the map induced on $\mathbb{X}$. A nonempty set $E\subset [0,1]$ is said to be \textit{completely invariant} under $f$, if $f(E)=E=f^{-1}(E)$. Similarly, a nonempty set $\hat{E}\subset\mathbb{X}$ is said to be \textit{completely invariant} under $\hat{f}$, if $\hat{f}(\hat{E}) = \hat{E} = \hat{f}^{-1}(\hat{E})$.
\begin{rem}\label{rem:cominv}
Let $E\subset (0,c)\cup(c,1)$ be a nonempty set. Then the set $E$ is completely invariant under $f$ if and only if the set $\pi^{-1}(E)$ is completely invariant under $\hat{f}$.
\end{rem}

A completely invariant closed set $D$ is \textit{minimal} if $D\subset E$ for every completely invariant closed set.
\begin{rem}
By the definition of completely invariant set  $E$ we should be able to equivalently write $f(E)\cup f^{-1}(E)\subset E$, because
$$
E\subset f^{-1}(f(E))\subset f^{-1}(E)\subset E.
$$
However deciding what are $f^{-1}(0)$ or $f^{-1}(1)$ is always problematic, and the above may easily lead to false statements. Therefore we will rather refer to map $\hat f$ when speaking about complete invariance,
since there all pre-images are clear (also for ``one-sided'' points $x_-$, $x_+$).
\end{rem}

For any nonempty open interval $U\subset [0,1]$, put 
\begin{equation}\label{def:N(U)}
N(U)=\min\{n\geq 0:f^n(z)=c\text{ for some }z\in U\}.
\end{equation}
It is clear by the definition that $c \in f^{N(U)}(U)$, so if $V\subset U$ then $N(U) \leq N(V )$ and $N(f^i(U))=N(U)-i$ for any $i=0, 1,\ldots , N(U)$.
Note that by the definition, $N(U)$ is usually the maximal integer such that $f^i$ is continuous on $U$ for $i=0,\ldots,N(U)$, where as usual we denote $f^0=\text{id}$.
The only situation that $f^i(U)$ does not lose continuity when passing through $c$ may happen when $f^i(U)\neq \Int f^i(U)$ and $c$ is at the boundary of this set.

Recall also that if $f(z)=c$ and $z\neq 0,1$ then $z$ splits into two points $z_-,z_+\in \mathbb{X}$.
Since formally, we have to define $f(c)=0$ or $f(c)=1$ the image of $f$ is always $[0,1)$ or $(0,1]$ (with the only possible exception when $f$ has a fixed point). A possible solution is to work with one-sided limits
$f(c_-)$ and $f(c_+)$, however we decided to use the space $\mathbb{X}$ for the clarity of the exposition and to avoid ambiguity.

Although we have already defined the Lorenz map as a map on the interval $[0,1]$, it is sometimes convenient to consider Lorenz maps on any interval $[a, b]$. Thus we say that $f$ is a Lorenz map on $[a, b]$, if taking the linear increasing homeomorphism $h\colon[a,b]\to[0,1]$ the composition $h\circ f\circ h^{-1}$ is a Lorenz map on $[0,1]$.

\begin{defn}\label{defn:renormalization}
Let $f$ be an expanding Lorenz map. If there is a proper subinterval $(u, v) \ni c$ of $(0,1)$ and integers $l, r > 1$ such that the map $g\colon [u, v]\to [u, v]$ defined by
	\begin{equation*}
	g(x)=\begin{cases}f^l (x),\,&\text{if 
		$x\in\left[u,c\right)$,}\\
	f^r (x),\,&\text{if $x\in (c,v],$}\end{cases}
	\end{equation*}
is itself a Lorenz map on $[u, v]$, then we say that $f$ is \emph{renormalizable} or that $g$ is a \emph{renormalization} of $f$. The interval $[u, v]$ is called the \emph{renormalization interval}. 

Furthermore, we say that an expanding Lorenz map $f$ is \emph{trivially renormalizable} if it satisfies the above definition with $(l,r)=(2,1)$ or $(l,r)=(1,2)$.
\end{defn}

\begin{rem}\label{rem:renorm}
The cases when $f^i(0)=c$ or $f^j(1)=c$ may cause problems with definition of renormalization, excluding some natural candidates.
For these situations we define $g(u)=f^l(u_+)$ and $g(v)=f^r(v_-)$.
\end{rem}

 Suppose $g = (f^m, f^k)$ is a renormalization map of an expanding Lorenz map $f$ with renormalization interval $[u, v] := [f^k(c_+), f^m(c_-)]$. Put
 \begin{equation}
 F_g =\{x\in [0,1]: \orb(x)\cap (u, v)\neq \emptyset\},\label{eq:Fg}
 \end{equation}
and
$$
 J_g =[0,1]\setminus F_g=\{x\in [0,1]: \orb(x)\cap (u, v)=\emptyset\}.
$$
The same way we define $\hat{F}_g, \hat{J}_g\subset \mathbb{X}$, using the map $\hat{f}$. In particular:
 \begin{equation}
\hat{F}_g =\{x\in \mathbb{X}: \orb(x)\cap (\hat{f}^k(c_+), \hat{f}^m(c_-))\neq \emptyset\}.\label{eq:hFg}
 \end{equation}

Before we proceed, let us make a simple observation.
$$
x\in\pi^{-1}(F_g)\iff \pi(x)\in F_g\iff\exists\ n\in\mathbb{N}_0:f^n(\pi(x))\in(u,v)
$$
Suppose $x\in\pi^{-1}(F_g)\setminus\hat{F}_g$. In particular $c_-$, $c_+\notin\orb(x)$, so $f^n(\pi(x))=\pi(\hat{f}^n(x))$. 
If we denote $\hat{u}=\hat{f}^k(c_+)$ and $\hat{v}=\hat{f}^m(c_-)$ then for each $n\geq 0$
$$
\hat{f}^n(x)\leq\hat{u}\quad\text{or}\quad\hat{f}^n(x)\geq\hat{v}
$$
and so
$$
f^n(\pi(x))=\pi(\hat{f}^n(x))\leq\pi(\hat{u})=u\quad\text{or}\quad f^n(\pi(x))=\pi(\hat{f}^n(x))\geq\pi(\hat{v})=v,
$$
which contradicts $x\in\pi^{-1}(F_g)$. 
Next, let $x\in\hat{F}_g$, i.e. $\hat{f}^n(x)\in(\hat{u},\hat{v})$ for some $n\in\mathbb{N}_0$. If $c_-\in\orb(x)$ or $c_+\in\orb(x)$, then $c\in\orb(\pi(x))$ and $\pi(x)\in F_g$. On the other hand $c_-$, $c_+\notin\orb(x)$ implies $f^n(\pi(x))=\pi(\hat{f}^n(x))$. Since $\hat{u}<\hat{f}^n(x)<\hat{v}$ 
we obtain $u<f^n(\pi(x))<v$. Therefore $\pi^{-1}(F_g)=\hat{F}_g$.
Summing up, we have the following.
\begin{rem}\label{rem:FgJg}
Since $c\in (u,v)$, there is almost no difference if we define $F_g$ for $f$ or $\hat{f}$, i.e. by the formula \eqref{eq:Fg} and then pull it back to $\mathbb{X}$,
or define $\hat{F}_g$ in $\mathbb{X}$ for the set $(\hat{f}^k(c_+), \hat{f}^m(c_-))$ and then project it into $[0,1]$.
Furthermore, we have the following equalities
$$
\hat{F}_g=\pi^{-1}(F_g)\quad\text{and}\quad\hat{J}_g=\pi^{-1}(J_g).
$$
\end{rem}

 Suppose $g = (f^m, f^k)$ is a renormalization map of an expanding Lorenz map $f$ with renormalization interval $[u, v] := [f^k(c_+), f^m(c_-)]$.
As we will see later (e.g. see Example~\ref{ex:4}), contrary to the statement in the proof of Theorem~A in \cite{Ding}, the set $F_g$ does not have to be completely invariant (this example is expressed in terms of $\hat{F}_g$, however it corresponds to $F_g$ in the language of \cite{Ding}). However, the following simple, yet useful statement holds:

\begin{lem}\label{lem:inv}
	Let $g = (f^m, f^k)$ be a renormalization map of an expanding Lorenz map $f$ and let $\hat{F}_g$ be defined by \eqref{eq:hFg}. 
	Denote $u=f^k(c_+)$, $v=f^m(c_-)$ and $\hat{u}=\hat{f}^k(c_+)$, $\hat{v}=\hat{f}^m(c_-)$
	Then the following conditions hold:
	\begin{enumerate}
		\item\label{lem:inv:1} The set $\hat{F}_g$ is completely invariant if and only if $\orb(u)\cap (u,v)\neq \emptyset$ and $\orb(v)\cap (u,v)\neq \emptyset$;
		\item\label{lem:inv:2}  If $\hat{F}_g$ is completely invariant, then the following assertions are true:
		\begin{enumerate}
			\item\label{lem:inv:2.1} $\bigcup_{n\geq 0}\hat{f}^n((\hat{u},\hat{v}))\subset \hat{F}_g$;
			\item\label{lem:inv:2.2}  for every $x\in[0,u)$, $y\in(v,1]$ we have $\orb(x)\cap (u,1]\neq \emptyset$ and $\orb(y)\cap [0,v)\neq \emptyset$ ;
			\item\label{lem:inv:2.3}  $f(x)\neq x$ for every $x\in [0,1]$;
			\item\label{lem:inv:2.4}  if $f(u)\in(u,v)$ or $f(v)\in(u,v)$, then $\hat{F}_g=\mathbb{X}$.
		\end{enumerate}
	\end{enumerate}
\end{lem}
\begin{proof}
	\eqref{lem:inv:1}:
	Directly from the definition we have
	$$
	\hat{f}^{-1}(\hat{F}_g)=\bigcup_{n\geq 0}\hat{f}^{-n-1}((\hat{u},\hat{v}))=\bigcup_{n> 0}\hat{f}^{-n}((\hat{u},\hat{v}))\subset \hat{F}_g
	$$
	so only forward invariance may be a problem.
	
	Let us first assume that there are $k_u,k_v>0$ such that $f^{k_u}(u)\in(u,v)$ and $f^{k_v}(v)\in(u,v)$. We also assume that $k_u,k_v$
	are the smallest such numbers.
	Fix any $x\in \hat{F}_g$ and $n\geq 0$ such that $\hat{f}^n(x)\in (\hat{u},\hat{v})$. If $c\in \orb(\pi(x))$
	then there is $s$ such that $\hat{f}^s(x)=\hat{u}$ or $\hat{f}^s(x)=\hat{v}$. But then $\hat{f}^{s+k_u}(x)\in  (\hat{u},\hat{v})$ or $\hat{f}^{s+k_v}(x)\in (\hat{u},\hat{v})$.
	This implies that $\hat{f}(x)\in \hat{F}_g$
	
	So let us assume that $c\not\in \orb(\pi(x))$.
	By definition $g(f^n(\pi(x)))\in [u,v]$. If $g(f^n(\pi(x)))=u$ then $f^{k_u}(g(f^n(\pi(x))))\in (u,v)$ and if $g(f^n(\pi(x)))=v$ then $f^{k_v}(g(f^n(\pi(x))))\in (u,v)$.
	In any case, there is $s>0$ such that $f^{n+s}(\pi(x))\in (u,v)$. But then $f(\pi(x))\in F_g$. 
	Note that if $c\not\in\orb(\pi(x))$ then $\pi^{-1}(f^i(\pi(x)))=\{\hat{f}^i(x)\}$
	and so clearly $\orb(\hat{f}(x))\cap (\hat{u},\hat{v})\neq \emptyset$, therefore $\hat{f}(x)\in \hat{F}_g$.
	
	In both cases $\hat{f}(x)\in \hat{F}_g$, showing that $\hat{f}(\hat{F}_g)\subset \hat{F}_g$.
	We have just proved that $\hat{F}_g$ is completely invariant.
	
	For the converse implication, let us assume that $\hat{F}_g$ is completely invariant. Then $c_+,c_-\in \hat{F}_g$ and as a result $0,1\in \hat{F}_g$, therefore also $\hat{u},\hat{v}\in \hat{F}_g$
	since $\hat{u}=\hat{f}^{k-1}(0)$ and $\hat{v}=\hat{f}^{m-1}(1)$. This implies that there are $i,j>0$ such that $\hat{f}^i(\hat{u}),\hat{f}^j(\hat{v})\in (\hat{u},\hat{v})$.
	Note that the situation $f(\pi(y))\neq \pi(\hat{f}(y))$ can only happen when $\pi(y)=c$. But this implies that $\orb(u)\cap (u,v)\neq \emptyset$ and $\orb(v)\cap (u,v)\neq \emptyset$.

\eqref{lem:inv:2}: 
Condition \eqref{lem:inv:2.1} is a straightforward consequence of the definition. Condition \eqref{lem:inv:2.2} is a simple consequence of monotonicity.
To see this, fix any $x\in(0,u)$, $y\in(v,1)$  and denote $i=N((0,x))$ and $j=N((y,1))$. By the definition of $N(U)$ we obtain that $f^i(x)>c>u$ and $f^j(y)<c<v$.
But if $x=0$ then similar argument to the above yields that $f^{k-1}(0)=u$ and the proof is completed by \eqref{lem:inv:1}.

For the proof of \eqref{lem:inv:2.3} note that since $f$ is expanding with two pieces of monotonicity, its graph can cross diagonal only at points $0,1$.
But if $0$ is fixed point then $u=f^{k-1}(0)=0$ and so $\orb(u)\cap (u,v)=\emptyset$ contradicting \eqref{lem:inv:1}. By symmetric argument $1$ is also not a fixed point.

Finally, for the proof of \eqref{lem:inv:2.4} assume that $f(u)\in (u,v)$. Then $\hat{f}(\hat{u})\in (\hat{u},\hat{v})$ and so
$$
(\hat{u},1]=(\hat{u},\hat{v})\cup(\hat{f}(\hat{u}),\hat{f}(c_-)]\subset(\hat{u},\hat{v})\cup \hat{f}((\hat{u},\hat{v}))\subset \hat{F}_g. 
$$
Next, fix any $\hat{x}\in [0,\hat{u})$. Then $x=\pi(\hat{x})\in [0,u]$. But then \eqref{lem:inv:2.2} implies that $\orb(u)\cap (u,1]\neq \emptyset$.
If $\orb(\hat{u})\cap \{c_-,c_+\}=\emptyset$ then we have $\orb(\hat u)\cap (\hat{u},1]\neq \emptyset$. But  $ \{c_-,c_+\}\subset (\hat u,1]$ so $\orb(\hat u)\cap (\hat{u},1]\neq \emptyset$ also when $\orb(\hat{u})\cap \{c_-,c_+\}\neq \emptyset$.
Since $\hat{F}_g$ is completely invariant, we obtain $\hat{F}_g=\mathbb{X}$. 
The case $f(v)\in (u,v)$ is dealt the same way.
The proof is completed.
\end{proof}

\begin{rem}\label{cor:Fg:leo}
If $f$ is an expanding Lorenz map and $g = (f^m, f^k)$ is a renormalization of $f$ then the set $F_g$ is dense. This follows directly from the fact that $f$ is expanding and $c\in (u,v)$.
\end{rem}

Renormalizations can be to some extent understood by analyzing kneading invariant, which we briefly recall below.
Let $f$ be an expanding Lorenz map and let $\hat{f}$ be the map induced on $\mathbb{X}$. For each $x\in\mathbb{X}$ we define the \textit{kneading sequence} $k(x)\in\{0,1\}^{\mathbb{N}_0}$ by
$$
k(x)_i=
\begin{cases}
1;&\text{if }\hat{f}^i(x)\geq c_+\\
0;&\text{if }\hat{f}^i(x)\leq c_-
\end{cases},
\quad\text{for}\quad i\in\mathbb{N}_0.
$$
The \textit{kneading invariant} of the map $f$ is the pair $k_f=(k_+,k_-)$, where $k_+=k(c_+)$ and $k_-=k(c_-)$.

Let $\sigma$ be the shift map defined on $\{0,1\}^{\mathbb{N}_0}$ and equip $\{0,1\}^{\mathbb{N}_0}$ with the lexicographic order, i.e. for $\xi=\xi_0\xi_1\ldots$, $\eta=\eta_0\eta_1\ldots\in\{0,1\}^{\mathbb{N}_0}$ we have $\xi\prec\eta$ if and only if there exists $i\in\mathbb{N}_0$ such that $\xi_{i}<\eta_{i}$ and $\xi_j=\eta_j$ for $j<i$. Moreover, $\xi\preccurlyeq\eta$ means $\xi\prec\eta$ or $\xi=\eta$. The main result in \cite{HubSpar} states the following.
\begin{thm}\label{thm:kneading}
	If $f$ is an expanding Lorenz map, then the kneading invariant $k_f=(k_+,k_-)$ satisfies
	\begin{equation}\label{eq:kneading}
	\sigma(k_+)\preccurlyeq\sigma^n(k_+)\prec\sigma(k_-)\quad\text{and}\quad\sigma(k_+)\prec\sigma^n(k_-)\preccurlyeq\sigma(k_-)
	\end{equation}
	for all $n\in\mathbb{N}$.
	
	Conversely, given any two sequences $k_+$ and $k_-$ satisfying (\ref{eq:kneading}), there exists an expanding Lorenz map $f$ with $(k_+,k_-)$ as its kneading invariant, and $f$ is unique up to conjugacy by a homeomorphism of $[0,1]$.
\end{thm}

A kneading invariant $k_f=(k_+,k_-)$ is \textit{renormalizable} if we can write it in the following form:
\begin{equation}
\label{eq:renormkneading1}
\begin{cases}
k_+=w_+(w_-)^\infty\quad\text{or}\quad k_+=w_+(w_-)^{p_1}(w_+)^{p_2}\ldots,\\
k_-=w_-(w_+)^\infty\quad\text{or}\quad k_-=w_-(w_+)^{m_1}(w_-)^{m_2}\ldots,
\end{cases}
\end{equation}
where $p_i$, $m_i\in\N$, $i=1,2,\ldots$ and the words $w_-\in\{0,1\}^l$, $w_+\in\{0,1\}^r$ have finite lengths $l,r>1$. We say that a kneading invariant $k_f=(k_+,k_-)$ is \textit{trivially renormalizable} if it can be written in the form~\eqref{eq:renormkneading1} with the words 
$$(w_+,w_-)=(10,0)\quad\text{or}\quad(w_+,w_-)=(1,01).$$

In \cite{GH} the authors claim that the kneading invariant of $f$ is renormalizable by words $w_-$ and $w_+$ of lengths $l$ and $r$ respectively if and only if the map $g=(f^l,f^r)$ is a renormalization of $f$ (see p.1004). As we will later see in Examples \ref{51from4} and \ref{ex:3}, this equivalence can hold only if we adopt definitions described in Remark \ref{rem:renorm}, since original values of $f$ defining $g$ may be inappropriate for the definition of Lorenz map.

It may happen that the kneading invariant of an expanding Lorenz map $f\colon[0,1]\to[0,1]$ is trivially renormalizable, but the corresponding map $g=(f^2,f)$ (or $g=(f,f^2)$, respectively) is not a trivial renormalization in the sense of the Definition~\ref{defn:renormalization}. This situation occurs when the point $0$ or $1$ is fixed for the map $f$. In this case the map $g\colon[u,v]\to[u,v]$ is a well-defined expanding Lorenz map, but the interval $[u,v]$ is not a proper subinterval of $[0,1]$. This type of maps is a special case of so-called \textit{special trivial renormalizations} or STR, which are considered in \cite{GS} (cf. \cite[Definition~1.3]{OPR}).
	
\begin{defn}\label{defn:STR}
Let $f\colon[0,1]\to[0,1]$ be an expanding Lorenz map with trivially renormalizable kneading invariant. If at least one of the following conditions holds:
$$
f(0)=0\quad\text{or}\quad f(0)=c\quad\text{or}\quad f(1)=1\quad\text{or}\quad f(1)=c,
$$
then we say that $f$ is \emph{special trivial renormalizable} (\emph{STR} for short). In this case the corresponding map $g=(f^2,f)$ (or $g=(f,f^2)$) is called a \emph{special trivial renormalization} of $f$.
\end{defn}

Let us note that the STR property is related to trivially renormalizable kneading invariants with $k_+=w_+(w_-)^\infty$ or $k_-=w_-(w_+)^\infty$. Special trivial renormalizations are important for our considerations because they can occur in the decomposition process of a renormalization, as the Example~\ref{51from4} illustrates.

\section{Results related to Theorem~\ref{thm:A}}\label{sec:ThmA}

Before proceeding further, let us recall Example~5.1 from \cite{OPR}.

\begin{exmp}\label{51from4}
	Let $f$ be the expanding Lorenz map
	$$
	f\colon [0,1]\ni x\mapsto h(x) (\text{mod }1) \in [0,1], 
	$$ 
 	where $h\colon[0,1]\to[0,2]$, $h(x)=\beta x+\alpha$ satisfies:
	\begin{enumerate}
		\item $h^4(0)=1$,
		\item $h(1)-1=h^2(0)$.
	\end{enumerate}
	These conditions lead to the equations
	\begin{equation}\label{eq:exBetaAlpha}
	\alpha=1-\frac{1}{\beta}\quad\text{and}\quad\beta^4-\beta-1=0,
	\end{equation}
	which are satisfied for $\beta_0\approx 1.2207440846$ and $\alpha_0=1-1/\beta_0\approx 0.1808274865$.
	Then $c=\frac{1}{\beta_0^2}$ and two orbits of $f$ are presented on
	Figure~\ref{fig_51}, where $q_i=f^{i-1}(1_-)$, $p_i=f^{i-1}(0_+)$ and $c$ is depicted as red dot.

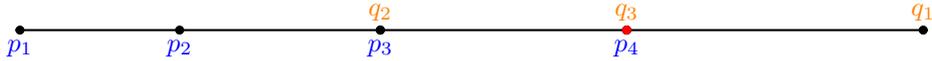
\begin{figure}[ht]
	\begin{tikzpicture}
	\draw[black, thick] (-6,0) -- (6,0);
	\filldraw [black] (12*0.005-6,0) circle (1.5pt) node[anchor=north] [blue] {$p_1$};
	\filldraw [black] (12*0.18-6,0) circle (1.5pt) node[anchor=north] [blue] {$p_2$};
	\filldraw [black] (12*0.40-6,0) circle (1.5pt) node[anchor=north] [blue] {$p_3$};
	\filldraw [black] (12*0.67-6,0) circle (1.5pt) node[anchor=north] [blue] {$p_4$};
	\filldraw [black] (12*0.995-6,0) circle (1.5pt) node[anchor=south] [orange] {$q_1$};
	\filldraw [black] (12*0.40-6,0) circle (1.5pt) node[anchor=south] [orange] {$q_2$};
	\filldraw [black] (12*0.67-6,0) circle (1.5pt) node[anchor=south] [orange] {$q_3$};
	\filldraw [red] (12*0.67-6,0) circle (1.5pt);
	\end{tikzpicture}
	\caption{Relation between points $p_i$, $q_i$ and $c$ from Example~\ref{51from4}.} 
	\label{fig_51}
\end{figure}
	
	It is not hard to see from Figure~\ref{fig_51} that $f$ is renormalizable with renormalization 
	\begin{equation*}
	g(x)=\begin{cases}f^4 (x),\,&\text{if 
		$x\in\left[f^2 (0),c\right)$,}\\
	f^3 (x),\,&\text{if $x\in (c,1],$}\end{cases}
	\end{equation*}
	and is locally eventually onto at the same time.
	
Next, observe that the kneading invariant of the map $f$ is
$$
k_{f}=\left((1000)^\infty,(010)^\infty\right)=\left(100(0100)^\infty,0100(100)^\infty\right).
$$
The renormalization $g=(f^4,f^3)$ corresponds to the renormalization of $k_f$ by words $w_-=0100$ and $w_+=100$. However, the kneading invariant $k_f$ can be also twice trivially renormalized by words $w_-^1=0$ and $w_+^1=10$. After the first trivial renormalization we obtain the map $f_1=(f,f^2)$ with the kneading invariant $k_{f_1}=((100)^\infty,(01)^\infty)$, and after the second one we get the map $f_2=(f_1,f_1^2)$ with $k_{f_2}=((10)^\infty,01^\infty)$. The kneading invariant $k_{f_2}$ is again trivially renormalizable by words $w_-^2=01$ and $w_+^2=1$, which leads to the map $f_3=(f_2^2,f_2)$ with $k_{f_3}=\left(10^\infty,01^\infty\right)$. Note that $g=f_3$, so the renormalization $g$ is the composition of two trivial and one special trivial renormalizations.
\end{exmp}

\begin{rem}
Note that the second equation of (\ref{eq:exBetaAlpha}) has two other roots $\beta_1$ and $\beta_2$ such that
$$
|\beta_1|=|\beta_2|\approx 1.06334,
$$
which implies that $\beta_0$ is algebraic but non-Pisot and non-Salem number. As we mentioned in introduction, it solves a question from \cite{Bru1}.
\end{rem}

\begin{rem}\label{rem:thma}
	Expanding Lorenz map in Example~\ref{51from4} is locally eventually onto however posses a renormalization (see also \cite[Remark~5.1]{OPR}).
	But if we denote by $g$ this renormalization, then $F_g=[0,1]$.
	It is a consequence of Theorem~\ref{thm:28} proved later in the paper, however can be easily verified directly.
	This shows that Theorem~\ref{thm:A}\eqref{thm:A:2} is not true.
	
	Let us also highlight the fact that by Corollary~\ref{cor:TransInv} this map does not have any proper completely invariant and closed subset.
\end{rem}

\begin{rem}\label{rem:thm:A}
	Theorem~\ref{thm:A}\eqref{thm:A:1} holds when $E\subset(0,c)\cup(c,1)$ since in this case we may view $E$ as an invariant set for $\hat{f}$,  see Remark~\ref{rem:cominv}.
\end{rem}

Lack of transitivity can be connected in Lorenz maps with so-called primary $n(k)$-cycles (see \cite{Glen}). A sequence $\{z_j : j\in\{0,\ldots, n-1\}\}$ is an \textit{$n(k)$-cycle} if its points satisfy:
\begin{enumerate}
\item $z_0 < z_1 <\dots < z_{n-k-1} < c < z_{n-k}<\dots <z_{n-1}$;
	\item $f(z_j)=z_{j+k (\text{mod }n)}$ for all $j=0,1,\ldots,n-1$;
	\item the integers $k$ and $n$ are coprime.
\end{enumerate}
If additionally
\begin{enumerate}\setcounter{enumi}{3}
	\item $z_{k-1}\leq f(0)$ and $f(1)\leq z_k$
\end{enumerate}
then the $n(k)$-cycle is \textit{primary}. Note that any $n(k)$-cycle is by definition an $n$-periodic orbit.
One of the main result of \cite{OPR} (cf. \cite{Glen}) states that expanding Lorenz map $f$ with primary $n(k)$-cycle is not transitive
provided that $f(0)\neq z_{k-1}$ or $f(1)\neq z_{k}$. This gives a chance that results of Theorem~\ref{thm:A}\eqref{thm:A:2} may hold at least partially. Appropriate result is stated below.
Let us also stress the fact, that even if $F_g$ is proper subset of $[0,1]$ it may happen that $R_{J_g}f\neq g$ because we cannot exclude the situation that a few different renormalizations define the same set $F_g$ (e.g. see Example~\ref{ex:3} below). In that sense, situation provided by Theorem~\ref{thm:cycle} is special.

\begin{thm}\label{thm:cycle}
	\label{tw-n(k)-cykle}
	Let $f$ be an expanding Lorenz map with a primary $n(k)$-cycle
	$$
	z_0<z_1<\dots<z_{n-k-1}<c<z_{n-k}<\dots<z_{n-1}.
	$$
Then the following conditions hold:
	\begin{enumerate}
		\item\label{tw-n(k)-cykle:1} the following $g\colon [u,v]\to[u,v]$ provided below is a well defined expanding Lorenz map which additionally is a renormalization of $f$:
		$$
		g(x)=\begin{cases}
		f^n(x); & x\in[u,c)\\
		f^n(x); & x\in(c,v]\\
		\end{cases},
		$$
		where $[u,v]:=[f^{n-1}(0),f^{n-1}(1)]$.
	\item\label{tw-n(k)-cykle:min} if $\tilde{g}=(f^l,f^r)$ is a renormalization of $f$ and at least one of the numbers $l$ and $r$ is greater or equal to $n$, then $n$ divides both $l$ and $r$.
		\item\label{tw-n(k)-cykle:2}\label{thm:cycle:2} if $z_{k-1}=f(0)$ or $z_k=f(1)$, then $\hat{F}_g$ is not completely invariant.
		\item\label{tw-n(k)-cykle:3} if $z_{k-1}\neq f(0)$ and $z_k\neq f(1)$ then:
		\begin{enumerate} 
			\item\label{tw-n(k)-cykle:3.1} $J_g$ is a completely invariant proper subset of $[0,1]$.
			\item\label{tw-n(k)-cykle:3.2} $z_{n-k-1} =\sup\{x\in J_g, x<c\}$ and $z_{n-k} =\inf\{x\in J_g, x>c\}$.
			\item\label{tw-n(k)-cykle:3.3} $R_{J_g}f=g$
		\end{enumerate}
	\end{enumerate}
\end{thm}
\begin{proof}
For a better clarity of the presentation, we will consider two cases.
First assume that $c\in [z_{k-1},z_k]$, which means that $n-k=k$, therefore $n=2k$. Since $n,k$ are co-prime, this is equivalent to saying that $k=1$. This means that $f(z_0)=z_1$ and $f(z_1)=z_0$. By definition $f(0)\geq z_0$. We are going to show that $f(0)<c$. If it is not the case, that is $f(0)\geq c$, then
$f((0,z_0))\subset (c,z_1)$ and consequently $f^2((0,z_0))\subset (0,z_0)$ and $f^2$ is monotone on $(0,z_0)$. This contradicts the fact that $f$ is an expanding Lorenz map. Using similar argument for $1$ we obtain that
$$
z_0\leq f(0)<c<f(1)\leq z_1
$$
and therefore, by monotonicity of $f$ on $[0,c)$ and $(c,1]$ we have $f^2(0)\geq z_1$, $f^2(1)\leq z_0$. This gives
\begin{eqnarray}
f ([u,c))&\subset & [f^2(0),1)\subset [z_1,1)\label{ucn2}\\
f ((c,v])&\subset & (0,f^2(1)]\subset (0,z_0].\label{vcn2}
\end{eqnarray}
This shows that $f^2$ is continuous and monotone on both intervals $[u,c)$ and $(c,v]$.
Observe that if $f^3(0)<u$ then $z_0\leq f^3(0)<u$ and so
$$
f^2([z_0,u])=f([z_1,f^2(0)])=[z_0,f^3(0)]\subset [z_0,u].
$$
This is a contradiction with the definition of expanding Lorenz map, since by the above $f$ and $f^2$ are monotone on $[z_0,u]$.
This shows that $f^3(0)\geq u$ and similarly $f^3(1)\leq v$. But then \eqref{ucn2} and \eqref{vcn2} lead to
$$
f^2([u,c))\subset [f^3(0),f(1))\subset [u,v),\quad f^2((c,v])\subset (f(0),f^3(1)]\subset (u,v].
$$
Then $g$ is well defined, while clearly also $g(c_-)=f(1)=v$ and $g(c_+)=f(0)=u$, completing the proof that $g$ is an expanding Lorenz map.
Indeed, $g$ is a renormalization of $f$.

Now consider the second case, $c\not\in [z_{k-1},z_k]$. This in particular means that $n>2$. By the definition of prime $n(k)$-cycle we obtain that 
\begin{eqnarray*}
z_{n-k-1}&=&f^n(z_{n-k-1})=f^{n-1}(z_{n-1})=f^{n-2}(z_{k-1}), \\
z_{n-k}&=&f^n(z_{n-k})=f^{n-1}(z_{0})=f^{n-2}(z_{k}).
\end{eqnarray*}
This shows that $f^{n-2}([z_{k-1},z_k])=[z_{n-k-1},z_{n-k}]\ni c$, and so $c\not\in f^{i}([z_{k-1},z_k])$ for $i=0,1,\ldots,n-3$. By the definition, we also have
$z_{k-1}\leq f(0)<f(1)\leq z_k$ and as a consequence.
\begin{equation}
\label{zk-k}
z_{n-k-1}\leq f^{n-1}(0)<f^{n-1}(1)\leq z_{n-k}.
\end{equation}
We claim that $f^{n-1}(0)<c<f^{n-1}(1)$. Assume first that $c\leq f^{n-1}(0)<z_{n-k}$. Then $f^n(0)<z_0=f^n(z_0)$ and so $z_0>0$.
Then $c\not\in f^i((0,z_0))$ for $i=0,\ldots,n$ and $f^n((0,z_0))\subset (0,z_0)$ which contradicts the fact that $f$ is expanding Lorenz map. By the same argument, condition
 $z_{n-k-1}< f^{n-1}(1)\leq c$ leads to a contradiction, either. Indeed, the claim holds.
Denote $u=f^{n-1}(0)$ and $v=f^{n-1}(1)$. By the previous observations
\begin{eqnarray*}
f([u,c))&=&[f^n(0),1)\subset [z_{n-1},1),\\
f([z_{n-1},1))&\subset& [z_{k-1},f(1))\subset [z_{k-1},z_k],
\end{eqnarray*}
and
\begin{eqnarray*}
	f((c,v])&=&(0,f^n(1)]\subset (0,z_0],\\
	f((0,z_{0}])&\subset& (f(0),z_{k}]\subset [z_{k-1},z_k].
\end{eqnarray*}
This shows that $f^n$ is continuous and monotone on both intervals $[u,c)$ and $(c,v]$. 
Next assume that $f^n(u)<u$. Clearly 
\begin{eqnarray}
f^n([u,c))&\subset& f^n([z_{n-k-1},c))\subset f^{n-1}([z_{n-1},1)) \nonumber\\
&\subset& f^{n-2}([z_{k-1},z_k]) \subset [z_{n-k-1},z_{n-1}] \label{ucfn}
\end{eqnarray}
so in particular $z_{n-k-1}\leq f^n(u)<u<c$. But this implies 
\begin{equation}
f^n((z_{n-k-1},u))\subset (z_{n-k-1},u)\label{znkuinv}
\end{equation}
and $c\not\in f^i((z_{n-k-1},u))$ for $i=0,\ldots,n$ which contradicts the assumption that $f$ is an expanding Lorenz map. This shows that $f^n(u)\geq u$.
By a symmetric argument we obtain that $f^n(v)\leq v$. But then monotonicity implies 
$f^n([u,c))\subset [f^n(u),f^{n-1}(1))\subset [u,v]$ and $f^n((c,v])\subset (f^{n-1}(0),f^n(v)]\subset [u,v]$. Directly from the definition we obtain that $f^n(c_+)=f^{n-1}(0)=u$
and $f^n(c_-)=f^{n-1}(1)=v$, so indeed $g$ is a well defined expanding Lorenz map. This proves \eqref{tw-n(k)-cykle:1}.

At first observe that the inequalities 
	$$z_{k-1}\leq f(0)<f(1)\leq z_k$$
imply $f^i(0)\neq f^i(1)$ for $i=0,1,\dots,n$ (see \eqref{zk-k}), which means that there is no matching for the first $n$ iterations. Now suppose that there is another renormalization $\tilde{g}=(f^l,f^r)$ on an interval $[\tilde{u},\tilde{v}]$ and assume that at least one of the numbers $l$ and $r$ is greater or equal to $n$. Say $r\geq n$, so $r=n\cdot p+i$ for some $p\in\mathbb{N}$ and $i\in\{0,1,\dots,n-1\}$. Next observe that for any point $x\in[z_{n-k-1},z_{n-k}]$ and any number $m\in\mathbb{N}$ we have $f^m(x)\in[z_{n-k-1},z_{n-k}]$ if and only if $n$ divides $m$. Consider the case $i\neq0$. Then we must have $\tilde{v}\in(c,z_{n-k}]$. Indeed, if $\tilde{v}>z_{n-k}$ then we obtain
$$
f^n((c,z_{n-k}))=(f^{n-1}(0),z_{n-k})\supset(c,z_{n-k})
$$
and
$$
c\in f^{n\cdot p}((c,z_{n-k}))\subset f^{n\cdot p}((c,\tilde{v})).
$$
It means that there is a point $y\in(c,\tilde{v})$ such that $f^{n\cdot p}(y)=c$. But then we get the contradiction with continuity of $f^r$, because there is no matching for the first $n$ iterations and hence
$$
\lim\limits_{x\to y-}f^r(x)=\lim\limits_{x\to y-}f^i(f^{n\cdot p}(x))\neq\lim\limits_{x\to y+}f^i(f^{n\cdot p}(x))=\lim\limits_{x\to y+}f^r(x).
$$
Note that condition ${\tilde{v}=f^{l}(c_-)}\in(c,z_{n-k}]$ implies that $l=q\cdot n$ for some $q\in\mathbb{N}$. Moreover $\tilde{u}\in[z_{n-k-1},c)$, because otherwise $\tilde{g}(z_{n-k-1})=f^l(z_{n-k-1})=z_{n-k-1}>\tilde{u}$ and we get the contradiction with topological expanding condition for $\tilde{g}$. But on the other hand ${\tilde{u}=f^{r}(c_+)}\in[z_{n-k-1},c)$ implies that $n$ divides $r$, which contradicts the assumption that $i\neq0$. Therefore the only possibility is $i=0$. In this case we also have $\tilde{v}\in(c,z_{n-k}]$, because otherwise $\tilde{g}(z_{n-k})=f^r(z_{n-k})=z_{n-k}<\tilde{v}$ and we get the contradiction with topological expanding condition for $\tilde{g}$. Finally ${\tilde{v}=f^{l}(c_-)}\in(c,z_{n-k}]$ implies that $n$ divides $l$.

Note that $\orb(z_0)\cap (u,v)=\emptyset$, so if assumptions of \eqref{tw-n(k)-cykle:2} hold, then $u\in \orb(z_0)$ or $v\in \orb(z_0)$ and then Lemma~\ref{lem:inv} implies that $\hat{F}_g$ is not invariant

Finally assume that assumptions in \eqref{tw-n(k)-cykle:3} are satisfied, that is $z_{k-1}\neq f(0)$ and $z_k\neq f(1)$. Then
\eqref{zk-k} implies
$$
z_{n-k-1}< f^{n-1}(0)<f^{n-1}(1)< z_{n-k}
$$
which in other words means that $(z_{n-k-1},u)$ and $(v,z_{n-k})$ are well defined nonempty intervals.
We claim that $g(u)\in (u,v)$. We already proved that $f^n(u)\in [u,v)$ so it remains to prove that $u$ cannot be a fixed point of $g$. But if $f^n(u)=u$
then repeating calculations in \eqref{ucfn} we obtain \eqref{znkuinv} which is impossible. Using symmetric argument, we obtain that $f^n(v)\neq v$, that is $g(v)\in (u,v)$.
Then Lemma~\ref{lem:inv} implies that $\hat{F}_g$ is completely invariant and hence the set $\hat{J}_g=\mathbb{X}\setminus\hat{F}_g$ is completely invariant as well. But as we mentioned earlier,  $\orb(z_0)\cap (u,v)=\emptyset$, therefore $\orb(z_0)\cap F_g=\emptyset$
showing that $F_g$ is a proper subset of $[0,1]$. Therefore condition (\ref{tw-n(k)-cykle:3.1}) follows from Remarks \ref{rem:cominv} and \ref{rem:FgJg}.

Recall that $\orb(z_0)\cap (u,v)=\emptyset$, so $\orb(z_0)\subset J_g$, hence if we denote
$$
e_- =\sup\{x\in J_g, x<c\}, 
\qquad e_+ =\inf\{x\in J_g, x>c\}
$$
then $[e_-,e_+]\subset [z_{n-k-1},z_{n-k}]$.
Therefore $[f(e_-),f(c_-)]\subset [z_{n-1},1]$, so 
$$
[f^2(e_-),f^2(c_-)]\subset [z_{k-1},f(1)]\subset [z_{k-1},z_k].
$$
It means that $l=N((e_-, c))\geq n$. If $l>n$ then $f^n(e_-)\in (c,f^{n-1}(1))$ which is a contradiction, since $J_g\cap (f^{n-1}(0),f^{n-1}(1))=\emptyset$.
Then $l=n$ and so $e_-=z_{n-k-1}$ because by Theorem~\ref{thm:A}(1) we have $f^n(e_-)=e_-$. Similar argument can be used to prove $e_+=z_{n-k}$, completing the proof of \eqref{tw-n(k)-cykle:3.2}.

Finally, let us apply Theorem~\ref{thm:A}(1) to $J_g$ (cf. Remark~\ref{rem:thm:A}). 
By \eqref{tw-n(k)-cykle:3.2} we obtain that $e_-=z_{n-k-1}$ and $e_+=z_{n-k}$. But then $N((e_-,c))=N((c,e_+))=n$ completing the proof.
\end{proof}

\begin{rem}
Condition \eqref{tw-n(k)-cykle:3} in Theorem~\ref{tw-n(k)-cykle} means that the map defined for $J_g$ by application of Theorem~\ref{thm:A}(1)
is exactly the renormalization provided by Theorem~\ref{tw-n(k)-cykle}.
\end{rem}

\begin{rem}
As we mentioned before, by results of \cite{OPR}, the case \eqref{thm:cycle:2} in Theorem~\ref{thm:cycle} within the class of Lorenz maps with constant slope $x\mapsto \alpha+\beta x$ includes the border case, when Lorenz map has renormalization but is transitive.
\end{rem}

It is also worth mentioning that primary $n(k)$-cycles are a special case of so-called twist orbits, which play an important role understanding and design of models based on Lorenz maps in applications (see \cite{Bart}  and references therein; cf. \cite{GellerMis}).

By Theorem~\ref{thm:cycle}\eqref{tw-n(k)-cykle:min} we know that if an expanding Lorenz map $f$ has a primary $n(k)$-cycle, then for any renormalization $\tilde{g}=(f^l,f^r)$ with $l\geq n$ or $r\geq n$, the number $n$ divides both $l$ and $r$. However, as the next example shows, the map $g=(f^n,f^n)$ does not need to be the minimal renormalization of $f$.
	\begin{exmp}
		\label{ex:5_2_cycle}
		Consider an expanding Lorenz map $f\colon [0,1]\to[0,1]$ defined by $f(x)=\beta x+\alpha (\text{mod }1),$ 
		where
		$$
		\beta:=\frac{9\sqrt[5]{2}}{10}\approx1.03383,\quad\alpha:=\frac{\sqrt[5]{2}}{3}\approx0.38289\quad\text{and}\quad c:=\frac{1-\alpha}{\beta}\approx0.59690.
		$$
	 Denote 
		$$
		z_0:=\frac{\alpha_0}{1-\beta^5},\quad\text{where}\quad\alpha_0:=\beta^4\alpha+\beta^3\alpha+\beta^2\alpha+\beta\alpha-\beta^2+\alpha-1.
		$$
		Then $z_0\approx0.11227$ and the orbit $O:=Orb(z_0)=\{z_0,z_1,z_2,z_3,z_4\}$ forms a primary $5(2)$-cycle for $f$. Let us denote $p_i=f^i(0_+)$ and $q_i=f^i(1_-)$, whose ordering in $[0,1]$ is depicted schematically on Figure~\ref{fig:5_2_cycle}. The critical point $c$ is marked as red dot.

	\begin{figure}[ht]
		\begin{tikzpicture}
		\draw[black, thick] (-6,0) -- (6,0);
		\filldraw [red] (12*0.596908-6,0) circle (1.5pt);
		\filldraw [blue] (12*0.382899-6,0) circle (1.5pt) node[anchor=north] {$p_1$};
		\filldraw [blue] (12*0.778752-6,0) circle (1.5pt) node[anchor=north] {$p_2$};
		\filldraw [blue] (12*0.187995-6,0) circle (1.5pt) node[anchor=north] {$p_3$};
		\filldraw [blue] (12*0.577254-6,0) circle (1.5pt) node[anchor=north] {$p_4$};
		\filldraw [orange] (12*0.416728-6,0) circle (1.5pt) node[anchor=south] {$q_1$};
		\filldraw [orange] (12*0.813725-6,0) circle (1.5pt) node[anchor=south] {$q_2$};
		\filldraw [orange] (12*0.224151-6,0) circle (1.5pt) node[anchor=south] {$q_3$};
		\filldraw [orange] (12*0.614633-6,0) circle (1.5pt) node[anchor=south] {$q_4$};
		\filldraw [black] (12*0.11227-6,0) circle (1.5pt) node[anchor=north] {$z_0$};
		\filldraw [black] (12*0.498967-6,0) circle (1.5pt) node[anchor=north] {$z_2$};
		\filldraw [black] (12*0.898746-6,0) circle (1.5pt) node[anchor=north] {$z_4$};
		\filldraw [black] (12*0.312048-6,0) circle (1.5pt) node[anchor=north] {$z_1$};
		\filldraw [black] (12*0.705504-6,0) circle (1.5pt) node[anchor=north] {$z_3$};
		\end{tikzpicture}
		\caption{Relation between points $p_i$, $q_i$, $z_i$ and $c$ from Example~\ref{ex:5_2_cycle}.} 
		\label{fig:5_2_cycle}
	\end{figure}
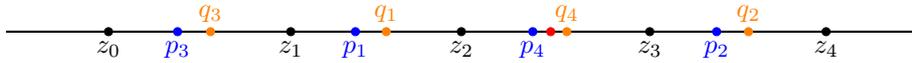
		
		By Theorem \ref{thm:cycle} the map
		$$
		g(x)=\begin{cases}
			f^5(x); & x\in[f^4(0),c)\\
			f^5(x); & x\in(c,f^4(1)]\\
		\end{cases},
		$$
		is the renormalization of $f$. Observe that $g=R_Df$  (cf. Corollary~\ref{cor:periodic}), where
		$$
		D:=\overline{\bigcup_{i=0}^\infty f^{-i}(O)}.
		$$
		But there exists another renormalization
		$$
		\tilde{g}(x)=\begin{cases}
			f^3(x); & x\in[f(0),c)\\
			f^2(x); & x\in(c,f^2(1)]\\
		\end{cases},
		$$
		so the renormalization $g$ associated to set $D$ is not the minimal renormalization of $f$. This example shows that the second statement in Theorem~\ref{thm:B}\eqref{thm:B:3} is also incorrect.
		
		Furthermore the renormalization $\tilde{g}$ is the composition of two trivial renormalizations, i.e. $f_1=(f,f^2)$, $f_2=(f_1^2,f_1)$ and $\tilde{g}=f_2$. Observe also that the map $f_2$ has renormalization $f_3=(f_2^2,f_2^2)$ and $f_3=g$.
	\end{exmp}

Next, we are going to provide two examples, which highlight further problems in Theorem~\ref{thm:A}(2). In fact, improvement of its statement to a correct and general form seems hard.
First, we are going to show, that some renormalizations $g$ cannot be recovered form associated sets $F_g$, even when they are proper subsets of $[0,1]$.
\begin{exmp}\label{ex:3}
We are going to construct a Lorenz map $f(x)=\beta x+\alpha(\text{mod }1)$ 
whose kneading invariant is 
\begin{equation}
\label{examp:kneading}
\left( k_+,k_-\right)=\left(k(c_+),k(c_-)\right)=\left(100101 (01100101)^\infty,01100101(100101)^\infty\right).
\end{equation}
Observe that such a map, among other things, must satisfy the conditions 
\begin{equation}
f^{13}(0_+)=f^5(0_+)\quad\text{and}\quad f^7(1_-)=f^{13}(1_-).\label{eq:13:5:7}
\end{equation}
To obtain the parameters $\beta$ and $\alpha$ we will use the formulas from \cite{DingSun} (see also \cite{Barnsley}). At first consider the power series $K(z)$ given by the kneading invariant~\eqref{examp:kneading}, i.e.
$$
\begin{aligned}
K\left(z\right)&=\sum_{i=0}^{\infty}\left(k(c_+)_i-k(c_-)_i\right)z^i \\
&=\left(1-z-z^2+z^3-z^{14}+z^{15}+z^{16}-z^{17}-z^{20}+z^{21} \right)\left(1+z^{24}+z^{48}+z^{72}+\ldots \right) \\
&=\left(1-z-z^2+z^3-z^{14}+z^{15}+z^{16}-z^{17}-z^{20}+z^{21} \right)\cdot\frac{1}{1-z^{24}}.
\end{aligned}
$$
Since the smallest positive zero of $K(z)$ is $z_0\approx0.905081$, by the result \cite[Lemma~7]{DingSun} we get
$$
\beta=\frac{1}{z_0}\approx1.104872.
$$
In order to compute the parameter $\alpha$ we use the formula from \cite[Theorem~2]{DingSun} for $x=0$. We obtain
$$
\begin{aligned}
\sum_{i=1}^{\infty}\frac{x_i}{\beta^i}&=\frac{1}{\beta^3}+\frac{1}{\beta^5}+\left( \frac{1}{\beta^7}+\frac{1}{\beta^8}+\frac{1}{\beta^{11}}+\frac{1}{\beta^{13}}\right)\left(1+\frac{1}{\beta^8}+\frac{1}{\beta^{16}}+\frac{1}{\beta^{24}}+\ldots\right)  \\
&=\frac{1}{\beta^3}+\frac{1}{\beta^5}+\left( \frac{1}{\beta^7}+\frac{1}{\beta^8}+\frac{1}{\beta^{11}}+\frac{1}{\beta^{13}}\right)\cdot\frac{1}{1-\frac{1}{\beta^8}}=\frac{1}{\beta^3}+\frac{1}{\beta^5}+\frac{1+\beta^2+\beta^5+\beta^6}{\beta^5\left(\beta^8-1\right) }\\
&=\frac{1+\beta+\beta^3+\beta^5}{\beta^8-1}.
\end{aligned}
$$
Hence
$$
\alpha=\left(\beta-1\right)\cdot\frac{1+\beta+\beta^3+\beta^5}{\beta^8-1}\approx0.438147
$$
and the approximate value of the critical point is
$$
c=\frac{1-\alpha}{\beta}\approx0.508522.
$$

Let us denote $p_i=f^i(0_+)$ and $q_i=f^i(1_-)$. Their order is depicted schematically on Figure~\ref{fig:ex:3} (the critical point $c$ is marked with a red dot). It is easy to verify that we were successful with our design and conditions in \eqref{eq:13:5:7} are satisfied.

\begin{figure}[ht]
	\begin{tikzpicture}
	\draw[black, thick] (-6,0) -- (6,0);
	\filldraw [black] (12*0.438148-6.48,0) circle (1.5pt) node[anchor=north] [blue] {$p_1$};
	\filldraw [black] (12*0.922245-6.3,0) circle (1.5pt) node[anchor=north] [blue] {$p_2$};
	\filldraw [black] (12*0.457112-6.3,0) circle (1.5pt) node[anchor=north] [blue] {$p_3$};
	\filldraw [black] (12*0.943198-6.1,0) circle (1.5pt) node[anchor=north] [blue] {$p_4$};
	\filldraw [black] (12*0.480262-6.15,0) circle (1.5pt) node[anchor=north] [blue] {$p_5$};
	\filldraw [black] (12*0.968776-5.9,0) circle (1.5pt) node[anchor=north] [blue] {$p_6$};
	\filldraw [black] (12*0.508522-6,0) circle (1.5pt) node[anchor=north] [blue] {$p_7$};
	\filldraw [black] (12*0.543021-5.9,0) circle (1.5pt) node[anchor=south] [orange] {$q_1$};
	\filldraw [black] (12*0.0381166-6.1,0) circle (1.5pt) node[anchor=south] [orange] {$q_2$};
	\filldraw [black] (12*0.480262-6.15,0) circle (1.5pt) node[anchor=south] [orange] {$q_3$};
	\filldraw [black] (12*0.968776-5.9,0) circle (1.5pt) node[anchor=south] [orange] {$q_4$};
	\filldraw [black] (12*0.508522-6,0) circle (1.5pt) node[anchor=south] [orange] {$q_5$};
	\filldraw [red] (12*0.508522-6,0) circle (1.5pt);
	\end{tikzpicture}
	\caption{Relation between points $p_i$, $q_i$ and $c$ from Example~\ref{ex:3}.} 
	\label{fig:ex:3}
\end{figure}
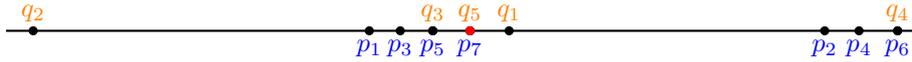

Observe that $\beta<\sqrt{2}$ and
$$
\frac{1}{\beta+\beta^2}\approx0.4299 \leq \alpha\leq\frac{-\beta^3+\beta^2+2\beta -1}{\beta+\beta^2}\approx0.4651,
$$
which by results of \cite[Theorem 6.6]{OPR} implies that $f$ has a primary $2(1)$-cycle. Furthermore, neither of points $f(0),f(1)$ is $2$-periodic, hence
$J_g$ associated to renormalization $g=(f^2,f^2)$ is a closed, fully invariant and proper subset of $[0,1]$ by Theorem~\ref{thm:cycle}. Note that the map $f$ has also a renormalization $\tilde{g}=(f^2,f^4)$ defined on $[\tilde{u},\tilde{v}]:=[f^3(0_+),f(1_-)]$.
Next, observe that
$$
[\hat{u},\hat{v}]:=[f^6(c_+),f^8(c_-)]=[f^5(0_+),f(1_-)]=[f^6(c_+),f^2(c_-)]
$$
so on $[\hat{u},\hat{v}]$ we have two well defined renormalizations $\hat{g}=(f^8,f^6)$ and $\overline{g}=(f^2,f^6)$.
Clearly
$(\hat{u},\hat{v})\subset(\tilde{u},\tilde{v})\subset(f(0),f(1))=(u,v)$, while
$$
f^4((u,\hat{u}))=(\hat{u},\hat{v})\quad \text{ and }\quad f^2(\hat{u})\in(\hat{u},\hat{v}),
$$
showing that $F_{\tilde{g}}=F_{\overline{g}}=F_{\hat{g}}= F_g$. 
In other words, all four renormalizations define the same completely invariant set $F_g$, while only $g$ can be recovered from $F_g$ by procedure presented in Theorem~\ref{thm:A}.

From the point of view of the kneading theory, \eqref{examp:kneading} can be renormalized by words $w_-=01$ and $w_+=10$, which leads to the kneading invariant of the following form
$$
k_{g}=\left(100(0100)^\infty,0100(100)^\infty\right).
$$
As we observed in Example~\ref{ex:3} the kneading invariant $k_g$ can be trivially renormalized three times, which corresponds to the renormalizations $\tilde{g}$, $\overline{g}$ and $\hat{g}$.
\end{exmp}

Next example shows that there may exist several fully invariant sets, defined by different renormalizations. This shows that it may be quite hard to
connect invariant sets with renormalizations in general.
\begin{exmp}\label{ex:4}
Consider an expanding Lorenz map $f\colon [0,1]\to[0,1]$,
$f(x)=\beta x+\alpha\mod 1,$
defined by
$$
\beta:=\sqrt[8]{2}\quad \text{ and }\quad\alpha:=\frac{2-\sqrt[8]{2}}{2}\approx0.4547
$$
Note that $c:=\frac{1-\alpha}{\beta}=\frac{1}{2}$
and denote $p_i=f^i(0_+)$, $q_i=f^i(1_-)$. Ordering of these points in $[0,1]$ is schematically depicted on Figure~\ref{fig:ex:4}.

\begin{figure}[ht]
	\begin{tikzpicture}
	\draw[black, thick] (-6,0) -- (6,0);
	\filldraw [black] (12*0.3125-6.1,0) circle (1.5pt) node[anchor=north] [blue] {$p_1$};
	\filldraw [black] (12*0.8125-6.1,0) circle (1.5pt) node[anchor=north] [blue] {$p_2$};
	\filldraw [black] (12*0.453125-6.15,0) circle (1.5pt) node[anchor=north] [blue] {$p_3$};
	\filldraw [black] (12*0.9375-5.9,0) circle (1.5pt) node[anchor=north] [blue] {$p_4$};
	\filldraw [black] (12*0.625-6.1,0) circle (1.5pt) node[anchor=north] [blue] {$p_5$};
	\filldraw [black] (12*0.125-6.1,0) circle (1.5pt) node[anchor=north] [blue] {$p_6$};
	\filldraw [black] (12*0.484375-6.1,0) circle (1.5pt) node[anchor=north] [blue] {$p_7$};
	\filldraw [black] (12*0.96875-5.8,0) circle (1.5pt) node[anchor=north] [blue] {$p_8$};
	\filldraw [black] (12*0.65625-6,0) circle (1.5pt) node[anchor=north] [blue] {$p_9$};
	\filldraw [black] (12*0.15625-6,0) circle (1.5pt) node[anchor=north] [blue] {$p_{10}$};
	\filldraw [black] (12*0.515625-5.9,0) circle (1.5pt) node[anchor=north] [blue] {$p_{11}$};
	\filldraw [black] (12*0.03125-6.2,0) circle (1.5pt) node[anchor=north] [blue] {$p_{12}$};
	\filldraw [black] (12*0.34375-6,0) circle (1.5pt) node[anchor=north] [blue] {$p_{13}$};
	\filldraw [black] (12*0.84375-6,0) circle (1.5pt) node[anchor=north] [blue] {$p_{14}$};
	\filldraw [black] (12*0.6875-5.9,0) circle (1.5pt) node[anchor=south] [orange] {$q_1$};
	\filldraw [black] (12*0.1875-5.9,0) circle (1.5pt) node[anchor=south] [orange] {$q_2$};
	\filldraw [black] (12*0.546875-5.85,0) circle (1.5pt) node[anchor=south] [orange] {$q_3$};
	\filldraw [black] (12*0.0625-6.1,0) circle (1.5pt) node[anchor=south] [orange] {$q_4$};
	\filldraw [black] (12*0.375-5.9,0) circle (1.5pt) node[anchor=south] [orange] {$q_5$};
	\filldraw [black] (12*0.875-5.9,0) circle (1.5pt) node[anchor=south] [orange] {$q_6$};
	\filldraw [black] (12*0.515625-5.9,0) circle (1.5pt) node[anchor=south] [orange] {$q_7$};
	\filldraw [black] (12*0.03125-6.2,0) circle (1.5pt) node[anchor=south] [orange] {$q_8$};
	\filldraw [black] (12*0.34375-6,0) circle (1.5pt) node[anchor=south] [orange] {$q_9$};
	\filldraw [black] (12*0.84375-6,0) circle (1.5pt) node[anchor=south] [orange] {$q_{10}$};
	\filldraw [black] (12*0.484375-6.1,0) circle (1.5pt) node[anchor=south] [orange] {$q_{11}$};
	\filldraw [black] (12*0.96875-5.8,0) circle (1.5pt) node[anchor=south] [orange] {$q_{12}$};
	\filldraw [black] (12*0.65625-6,0) circle (1.5pt) node[anchor=south] [orange] {$q_{13}$};
	\filldraw [black] (12*0.15625-6,0) circle (1.5pt) node[anchor=south] [orange] {$q_{14}$};
	\filldraw [red] (12*0.5-6,0) circle (1.5pt);
	\end{tikzpicture}
	\caption{Relation between points $p_i$, $q_i$ and $c$ from Example~\ref{ex:4}.} 
	\label{fig:ex:4}
\end{figure}
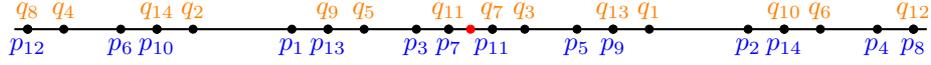

Simple calculations yield that $f^{15}(0)=f^7(0)$ and $f^{15}(1)=f^7(1)$. 
It is not hard to see from Figure~\ref{fig:ex:4} that if we denote
\begin{eqnarray*}
[u,v]&:=&[f(0),f(1)]=[p_1,q_1],\\
{[\hat{u},\hat{v}]}&:=&[f^3(0),f^3(1)]=[p_3,q_3],\\
{[\overline{u},\overline{v}]}&:=&[f^7(0),f^7(1)]=[p_7,q_7],
\end{eqnarray*}
then
$$
		g(x)=\begin{cases}
		f^2(x); & x\in[u,c)\\
		f^2(x); & x\in(c,v]\\
		\end{cases},\quad
		\hat g(x)=\begin{cases}
		f^4(x); & x\in[\hat{u},c)\\
		f^4(x); & x\in(c,\hat{v}]\\
		\end{cases},\quad
		\overline{g}(x)=\begin{cases}
		f^8(x); & x\in[\overline{u},c)\\
		f^8(x); & x\in(c,\overline{v}]\\
		\end{cases}
$$
are well defined renormalizations of $f$.
Since $\beta =\sqrt[8]{2}<\sqrt{2}$ and
$$
\frac{1}{\beta+\beta^2}\approx0.4386 \leq \alpha\leq\frac{-\beta^3+\beta^2+2\beta -1}{\beta+\beta^2}\approx0.4708,
$$
the map $f$ has primary $2(1)$-cycle by \cite[Theorem 6.6]{OPR} and map $g$ with completely invariant set $J_g$ is provided by Theorem~\ref{thm:cycle}. This also means that $F_g$ is a proper subset of $[0,1]$.
On the other hand, orbits of points $\overline{u}, \overline{v}$  never intersect $(\overline{u},\overline{v})$ so $\hat{F}_{\overline{g}}$ is not completely invariant by Lemma~\ref{lem:inv}.
By the same argument, since $[\overline{u},\overline{v}]\subset (\hat u, \hat v)$ the set $\hat{F}_{\hat g}$  is completely invariant.
We claim that $F_g\neq F_{\hat{g}}$. Observe that
$$
f([q_3,p_5]) =[q_4,p_6],\hspace{0.2cm}f([q_4,p_6])=[q_5,p_7],\hspace{0.2cm}f([q_5,p_7])=[q_6,p_8],\hspace{0.2cm}f([q_6,p_8])=[q_7,p_9],
$$
so $[q_3,p_5]\subset f^4([q_3,p_5]).$ This means that there is a $4$-periodic point $z\in [q_3,p_5]$ and since none of the points $q_3,p_5$ is $4$-periodic, we must have
$z\in [u,v]\setminus [\hat{u},\hat{v}]$. It is also clear that $\orb(z)\cap [\hat{u},\hat{v}]=\emptyset$ so indeed the claim holds.
\end{exmp}

\section{Results related to Theorem~\ref{thm:B}}\label{sec:thmB}

The following Remark~\ref{rem:thmb} shows that Theorem~\ref{thm:B} is not true as stated.
The aim of this section will be investigation how many of claims in Theorem~\ref{thm:B} hold.

\begin{rem}\label{rem:thmb}
	Expanding Lorenz map in Example~\ref{51from4} does not have fixed points and is renormalizable. By the definition it also does not have points of period $2$. 
	It has however, a point of period $3$ containing $c_-$ in its orbit according to terminology in \cite{Ding} (cf. \cite[Remark 1]{Ding}; in terms of map $\hat{f}$ point $c_-$ is a point of period $3$, while for $f$ point $c$ has period $4$).
Observe that $D$
	in Theorem~\ref{thm:B} satisfies $D=[0,1]$, showing that Theorem~\ref{thm:B}\eqref{thm:B:3} is false (here $D$ is calculated using $\hat{f}$ and then $\pi$; using $f$ may cause problems, but it seems it is not intention in \cite{Ding}).
\end{rem}

The main difficulty in analysis of Theorem~\ref{thm:B} is that huge part of the proof in \cite{Ding} is based on the following lemma, which is not true in general as shown in Example~\ref{lem:tBce}.
\begin{lem}\label{lem32Ding}
Let $f$ be an expanding Lorenz map, and $1 < \kappa < \infty$ be the smallest period of the periodic points of $f$. Suppose $O$ is a $\kappa$-periodic orbit,
$P_L =\max\{x\in O,x<c\}$, $P_R =\min\{x\in O,x>c\}.$ 
Then we have 
$$
N((P_L, c)) = N((c, P_R)) = \kappa.
$$
\end{lem}
\begin{exmp}\label{lem:tBce}
Consider expanding Lorenz map $f$ from Example~\ref{51from4}. Then 
$$
O:=Orb(c_-)=\{c_-,1,f(1)\},
$$
is periodic orbit of $f$ (in the terminology of \cite{Ding}; and of $\hat{f}$ in the terminology of the present paper) with the smallest period, and hence $\kappa=3$. By the definition in Lemma~\ref{lem32Ding} we also have
$$
P_L:=\max\{x\in O\hspace{0.2cm}|\hspace{0.2cm}x<c\}=f(1),\hspace{0.2cm}P_R:=\min\{x\in O\hspace{0.2cm}|\hspace{0.2cm}x>c\}=1,
$$
however
$$
\begin{aligned}
f((P_L,c))&=(c,1),\hspace{0.2cm}&f((c,1))=(0,f(1)),\\f((0,f(1)))&=(f(0),c),\hspace{0.2cm}&f((f(0),c))=(f(1),1),\\
\end{aligned}
$$
which clearly means that $N((P_L,c))=4$. 
\end{exmp}

It is not hard to provide an example of continuous map $f$ on $[0,1]$ such that for some $x$ the set $Q=\overline{\bigcup_{k=0}^\infty f^{-k}(x)}$
satisfies $f^{-1}(Q)\setminus Q\neq \emptyset$. Next lemma shows that in expanding Lorenz maps such situation never happens. 

\begin{lem}
	\label{lem:PreimInv}
	Let $f$ be an expanding Lorenz map, $\hat f$ be its extension to the Cantor set $\mathbb{X}$ and $x\in\mathbb{X}$. Denote $D=\overline{\bigcup_{k=0}^\infty\hat{f}^{-k}(x)}$. Then
	$$
	\hat{f}^{-1}(D)\setminus D\subset\{0,1\}.
	$$
	Moreover, if the map $f$ is transitive, then the set $D$ is backward invariant.
\end{lem}
\begin{proof}
At first, observe that for any point $y\in \mathbb{X}\setminus\{0,1\}$ and its open neighborhood $U$ there is an open set $V\subset U$ such that $y\in V$ and $\hat{f}(V)$ is also open in $\mathbb{X}$. Indeed, if $y\neq c_-$ and $y\neq c_+$, then there is an interval $y\in (a,b)\subset U$ such that $\hat{f}((a,b))=(\hat{f}(a),\hat{f}(b))$. Otherwise, $y=c_-\in (a,c_-]\subset U$ for some $a\in\mathbb{X}$ or $y=c_+\in [c_+,b)\subset U$ for some $b\in\mathbb{X}$. Clearly the sets $\hat{f}((a,c_-])=(\hat{f}(a),1]$ and $\hat{f}([c_+,b))=[0,\hat{f}(b))$ are both open.

Now suppose that there exists a point $y\in\hat{f}^{-1}(D)$ such that $y\notin\{0,1\}$ and $y\notin D$. So, for some open neighborhood $U$ of the point $y$ we have $U\cap\bigcup_{k=0}^\infty\hat{f}^{-k}(x)=\emptyset$. By the previous observation, there is an open set $V\subset U$ such that $\hat{f}(V)$ is an open neighborhood of the point $\hat{f}(y)$. But then $\hat{f}(V)\cap\bigcup_{k=0}^\infty\hat{f}^{-k}(x)=\emptyset$ which contradicts $y\in\hat{f}^{-1}(D)$. Therefore $\hat{f}^{-1}(D)\setminus D\subset\{0,1\}$.
	
Next, let $0\in\hat{f}^{-1}(D)\setminus D$ (the case $1\in\hat{f}^{-1}(D)\setminus D$ is dealt similarly) and let $V$ be an open neighborhood of the point $0$ such that $V\cap\bigcup_{k=0}^\infty\hat{f}^{-k}(x)=\emptyset$. Let us denote
$$
k_0:=\min\{n\in\mathbb{N}_0: c_-\in\hat{f}^n(V)\}.
$$
Then there is an open set $U\subset\hat{f}^{k_0+1}(V)$ containing the point $1$. Clearly, we have $U\cap\bigcup_{k=0}^\infty\hat{f}^{-k}(x)=\emptyset$. Hence,
$$
[0,z_1)\cap\bigcup_{k=0}^\infty f^{-k}(\pi(x))=\emptyset\quad\text{and}\quad(z_2,1]\cap\bigcup_{k=0}^\infty f^{-k}(\pi(x))=\emptyset
$$
for some points $z_1,z_2\in(0,1)$. Denote
$$
z_L:=\sup\left\lbrace z\in(0,1):[0,z)\cap\bigcup_{k=0}^\infty f^{-k}(\pi(x))=\emptyset\right\rbrace,\quad k_L:=N([0,z_L)),
$$
$$
z_R:=\inf\left\lbrace z\in(0,1):(z,1]\cap\bigcup_{k=0}^\infty f^{-k}(\pi(x))=\emptyset\right\rbrace,\quad k_R:=N((z_R,1])
$$
and
$$
W:=\bigcup_{i=0}^{k_L}f^i([0,z_L))\cup\bigcup_{i=0}^{k_R}f^i((z_R,1]).
$$
Note that by the definition of $z_L$ and $z_R$ we have
$$
W\cap\bigcup_{k=0}^\infty f^{-k}(\pi(x))=\emptyset\quad\text{and}\quad f(W)\subset W.
$$
Since $\hat{f}(0)\in D$ and $0\notin D$, there is a strictly increasing sequence $\{x_n\}_{n=1}^\infty\subset\bigcup_{k=0}^\infty f^{-k}(\pi(x))$ such that $\lim\limits_{n\to\infty}x_n=f(0)$. Let $I_n:=(x_n,x_{n+1})$ for each $n\in\mathbb{N}$. Observe that $\{I_n\}_{n=1}^\infty$ is an infinite sequence of pairwise disjoint intervals and $W$ is a union of finite number of intervals. Therefore $W\cap I_{n_0}=\emptyset$ for some $n_0\in\mathbb{N}$ and, by forward invariance, $f^n(W)\cap I_{n_0}=\emptyset$ for every $n\in\mathbb{N}$ which is impossible in the case of transitive map $f$. 
\end{proof}
	
	\begin{thm}\label{thm:Bimprov}
		Let $f$ be an expanding Lorenz map and let $\hat f$ be its extension to the Cantor set $\mathbb{X}$. Let $1 < \kappa < \infty$ be the smallest period of the periodic points of $\hat{f}$. Suppose $O$ is a $\kappa$-periodic orbit.
		The following assertions hold:
		\begin{enumerate}
			\item\label{thm:Bimprov:1} The set $D_O=\overline{\bigcup_{n\geq0}\hat{f}^{-n}(O)}$ is forward invariant and satisfies
			$$
			 \hat{f}^{-1}(D_O)\setminus D_O\subset\{0,1\}.
			$$
			\item\label{thm:Bimprov:2} $O$ is the unique $\kappa$-periodic orbit of $\hat{f}$.
		\end{enumerate}
	\end{thm}
	\begin{proof}
		\eqref{thm:Bimprov:1}:
		The inclusion $\hat{f}^{-1}(D_O)\setminus D_O\subset\{0,1\}$ follows from the Lemma \ref{lem:PreimInv}. In order to prove the forward invariance let us denote $A:=\bigcup_{n\geq0}\hat{f}^{-n}(O)$. Then by continuity of $\hat{f}$ we obtain
		$$
		\hat{f}(\overline{A})\subset\overline{\hat{f}(A)}=\overline{\hat{f}(\bigcup_{n\geq0}\hat{f}^{-n}(O))}=\overline{A\cup O}=\overline{A}.
		$$
		
		\eqref{thm:Bimprov:2}:
		Suppose $\widetilde{O}$ is a $\kappa$-periodic orbit and $\widetilde{O}\neq O$. Denote
		$$
		P_L:=\max\{x\in O:x\leq c_-\}\quad\text{and}\quad \widetilde{P}_L:=\max\{x\in \widetilde{O}:x\leq c_-\}.
		$$
		Since $P_L\neq\widetilde{P}_L$ we can assume that $P_L<\widetilde{P}_L$ and denote $l:=N((\pi(P_L),\pi(\widetilde{P}_L)))$. We claim  that the number $l$ is well defined. Assume that $\pi(P_L)=\pi(\widetilde{P}_L)$. Since  $\widetilde{O}\neq O$, then one of these sets must contain $c_-$, in particular $\widetilde{P}_L=c_-$. But $\pi(P_L)=\pi(\widetilde{P}_L)=c$, thus the only possibility is that $P_L=c_+$ which is impossible. Indeed, the claim holds.
		
		Assume first that $c_+\notin\widetilde{O}$ and let us consider two cases:
		
		\textbf{Case 1}. $l<\kappa$. By the definition of $l$, the map $\hat{f}^l$ is monotone on $[P_L,\widetilde{P}_L]$. On the other hand $c\in f^l((\pi(P_L),\pi(\widetilde{P}_L)))$,
		so $\hat{f}^l(P_L)<c_-<c_+< \hat{f}^l(\widetilde{P}_L)$. By the definition $\hat{f}^l(P_L)< P_L$ which gives $[P_L,\widetilde{P}_L]\subset \hat{f}^l([P_L,\widetilde{P}_L])$.

	Then there is $[a,b]\subset (P_L,\widetilde{P}_L)$ such that $\hat{f}^l([a,b])=[P_L,\widetilde{P}_L]$.		
	Then $[\pi(a),\pi(b)]\subset f^l([\pi(a),\pi(b)])$ and $f^l$ is continuous on that interval, and so there is a point $x\in [a,b]$ 
	with $f^l(x)=x$. Note that $f^i(x)\neq c$ for $i=0,1,\ldots, l$ by the definition of $l$, so $\pi^{-1}(\{x\})$ is a singleton, and so $\hat{f}$ has point of period $l<\kappa$. It is a contradiction.

\textbf{Case 2}.			
			 $l\geq\kappa$. Then, by periodicity of $P_L$ and $\widetilde{P}_L$ we obtain
			$$
			\hat{f}^\kappa((P_L,\widetilde{P}_L))=(\hat{f}^\kappa(P_L),\hat{f}^\kappa(\widetilde{P}_L))=(P_L,\widetilde{P}_L),
			$$
			which contradicts the topological expanding condition.

Next assume that $c_+\in\widetilde{O}$. Then $O\cap \{c_-,c_+\}=\emptyset$.
If we denote
		$$
		P_R:=\min\{x\in O:x\geq c_+\}\quad\text{and}\quad \widetilde{P}_R:=\min\{x\in \widetilde{O}:x\geq c_+\},
		$$
we obtain that $c_+=\widetilde{P}_R<P_R$ and appropriate modifications of the cases above completes the proof.
	\end{proof}

	\begin{lem}
		\label{lem:expan}
		For any two points $x,y\in\mathbb{X}$ with $d(x,y)<1$ we have $d(\hat{f}(x),\hat{f}(y))\geq d(x,y)$.
	\end{lem}
	\begin{proof}
		Let $x,y\in\mathbb{X}$ and assume $x<y$ (for $x=y$ the statement holds). First, note that $d(c_-,c_+)=1$. So $d(x,y)<1$ implies that $x,y\in[0,c_-]$ or $x,y\in[c_+,1]$. By expanding condition, if $x,y\in[0,c_-)$ or $x,y\in(c_+,1]$, then
		$$
		|\pi(\hat{f}(x))-\pi(\hat{f}(y))|=|f(\pi(x))-f(\pi(y))|\geq|\pi(x)-\pi(y)|.
		$$
		Similarly, if $y=c_-$, then
		$$
		|\pi(\hat{f}(x))-\pi(\hat{f}(y))|=|f(\pi(x))-1|>|\pi(x)-c|
		$$
		 and in the case $x=c_+$ we have
		$$
		|\pi(\hat{f}(x))-\pi(\hat{f}(y))|=|0-f(\pi(y))|>|c-\pi(y)|.
		$$
		Moreover, observe that $N(x,y)=N(\hat{f}(x),\hat{f}(y))+1$ provided that $x,y\in[0,c_-]$ or $x,y\in[c_+,1]$. This applied to \eqref{eq:metricX} implies that $d(\hat{f}(x),\hat{f}(y))\geq d(x,y)$ whenever $d(x,y)<1$.
	\end{proof}

	\begin{thm}
		\label{thm:TransDenseOrbits}
		Let $f$ be a transitive and expanding Lorenz map and $\hat{f}$  be the map induced on $\mathbb{X}$. Then for every $x\in\mathbb{X}$ the set $\bigcup_{k=0}^\infty\hat{f}^{-k}(x)$ is dense in $\mathbb{X}$.
	\end{thm}
	\begin{proof}
		The proof presented below is a modification of the proof of Theorem 1 from \cite{Kam}. In our proof we do not use the openness of the map $\hat{f}$, which in fact often does not hold.
This always happens, when $0$ is transformed into a point $a\in \mathbb{X}$ such that $a>0$ and $\pi^{-1}(\pi(a))$ is a singleton.
		
		Suppose that there is a point $a\in\mathbb{X}$ for which
		$$
		S:=\overline{\bigcup_{k=0}^\infty\hat{f}^{-k}(a)}\neq\mathbb{X}.
		$$
		By Lemma \ref{lem:PreimInv} we have $\hat{f}^{-1}(S)\subset S$ and therefore $\hat{f}(\mathbb{X}\setminus S)\subset\mathbb{X}\setminus S$. Since $\mathbb{X}\setminus S$ is an open set and $\hat{f}$ is a transitive map, there is a point $x\in\mathbb{X}\setminus S$ with dense orbit. Note that $d(\hat{f}^k(x),S)>0$ for any $k\geq0$ 
while $\liminf_{k\to\infty}d(\hat{f}^k(x),S)=0.$
Therefore for every $i\in\mathbb{N}$ the following number is well defined
		$$
		n(i):=\min \left\lbrace j\in\mathbb{N} : d(\hat{f}^j(x),S)\leq\frac{1}{i}\right\rbrace
		$$
forming a nondecreasing sequence and $n(i)>1$ for all $i$ sufficiently large.
		There is a subsequence $\{\hat{f}^{n(i_j)}(x)\}_{j=0}^\infty$ of the sequence $\{\hat{f}^{n(i)}(x)\}_{i=1}^\infty$ which converges to some point $y\in S$. 
		Note that $n(i)$ is unbounded, as otherwise $\hat{f}^{n(i)}(x)\in S$ for some $i$, which contradicts backward invariance of $S\not\ni x$.
		Going to a subsequence when necessary, we may assume that $z=\lim_{j\to\infty}\hat{f}^{n(i_j)-1}(x)$ exists, and clearly $\hat{f}(z)=y$.
For every $i$ let $x_i\in S$ be a point such that 
$$
d(\hat{f}^{n(i)}(x),x_i)=d(\hat{f}^{n(i)}(x),S).
$$ 
Clearly $\hat{f}^{-1}(x_i)\cap \{0,1,c_-,c_+\}=\emptyset$ as otherwise by backward invariance, $S$ is dense, which is a contradiction.
By the same argument $y,z\not \in \{0,1,c_-,c_+\}$. Then there exists an open neighborhood $V\ni z$ such that $\pi(V)$ is an open interval, and therefore
$\hat{f}(V)\ni y$ is an open set and $\hat{f}|_V$ is injective. But clearly $\lim_{j\to \infty}x_{i_j}=\lim_{j\to \infty}\hat{f}^{n(i_j)}(x)=y$
and so we may assume that $x_{i_j}\in \hat{f}(V)$ for every $j$. Hence, there exists a unique $q_{i_j}\in V$ such that $\hat{f}(q_{i_j})=x_{i_j}$.
We may also assume, going to a subsequence if necessary,  that $\lim_{j\to \infty}q_{i_j}$ exists, and thus it must by $z$
since $\hat{f}(\lim_{j\to \infty}q_{i_j})=\lim_{j\to \infty}x_{i_j}$ and $\hat{f}$ is injective on $V$.
In other words
$$
\lim_{j\to \infty}q_{i_j}=\lim_{j\to\infty}\hat{f}^{n(i_j)-1}(x).
$$
By backward invariance, we also have that $q_{i_j}\in S$, and then
by expanding condition provided by Lemma~\ref{lem:expan} and by the definition of $n(i)$, for all $j$ sufficiently large we have 
$$
1/i_j<d(\hat{f}^{n(i_j)-1}(x),q_{i_j})\leq d(\hat{f}^{n(i_j)}(x),x_{i_j})\leq 1/i_j
$$
which is a contradiction.
\end{proof}

The following result shows that there is no good candidate for set $D$ in Theorem~\ref{thm:B}\eqref{thm:B:3}
in the case of transitive maps.

	\begin{cor}\label{cor:transitive}
		If $f$ is a transitive and expanding Lorenz map, then for any periodic orbit $O$  of the map $\hat{f}$ we have 
		$$
		D_O=\overline{\bigcup_{n\geq0}\hat{f}^{-n}(O)}=\mathbb{X}.
		$$
	\end{cor}

\begin{cor}\label{cor:TransInv}
	If $f$ is a transitive and expanding Lorenz map without fixed points, then any proper closed set in $[0,1]$ is not completely invariant for $f$.
\end{cor}
\begin{proof}
	Let $\hat{f}$ be the map induced on $\mathbb{X}$ and suppose there exists a proper and closed subset $B\subset[0,1]$ which is completely invariant under $f$. 
Since $f$ is expanding, $\bigcup_{k=0}^\infty f^{-k}(c)$ is dense in $[0,1]$ and therefore
$c\notin B$. Since $0,1$ are not fixed points, then their preimage is either empty or contains $c$.
In any case $0,1\notin B$ because otherwise either $B=[0,1]$ or $f(B)\neq B$.

By Remark~\ref{rem:cominv} the set $\hat{B}=\pi^{-1}(B)$ is completely invariant under $\hat{f}$. By continuity of the projection $\pi$ the set $\hat{B}$ is also closed,
which by Theorem~\ref{thm:TransDenseOrbits} yields that $\hat{B}=\mathbb{X}$ and consequently, by surjectivity of $\pi$, we have $[0,1]=\pi(\mathbb{X})=\pi(\pi^{-1}(B))= B$.
This is a contradiction, completing the proof.
\end{proof}

The following is an adaptation of Lemma~4.1 in \cite{Ding} to our context. The proof follows the same lines. We present it for the reader convenience.

\begin{lem}\label{lem:41Dingimproved}
Let $f$ be an expanding Lorenz map and let $\hat{f}$ be the map induced on $\mathbb{X}$. Denote
$$
m:=\min\{i\in\mathbb{N}_0:\hat{f}^{-i}\left(\{c_-,c_+\} \right) \cap[\hat{f}(0),\hat{f}(1)]\neq\emptyset\}.
$$
Then $\hat{f}^{(m+2)}$ has a fixed point.
\end{lem}	
\begin{proof}
At first, observe that $\hat{f}(0)<\hat{f}(1)$ and $m$ is well-defined natural number, because $f$ is expanding. Moreover, note that there are two different points $z_0$, $z_1\in\mathbb{X}$ such that $\hat{f}(z_0)=z=\hat{f}(z_1)$ if and only if $z\in[\hat{f}(0),\hat{f}(1)]$. By the definition of $m$ we get $\hat{f}^{-m}(c_-)\cap[\hat{f}(0),\hat{f}(1)]\neq \emptyset$ or $\hat{f}^{-m}(c_+)\cap [\hat{f}(0),\hat{f}(1)]\neq \emptyset$. 
Let us assume that the first case holds. Then $|\hat{f}^{-i}(c_-)|=1$ for $i=0,1,\ldots,m$ and $|\hat{f}^{-(m+1)}(c_-)|=2$. Denote the unique point $\{c_i\}=\hat{f}^{-i}(c_-)$ for $i=0,1,\ldots,m$, let $\hat{f}^{-(m+1)}(c_-)=\{c_{m+1},c'_{m+1}\}$ with $c_{m+1}<c'_{m+1}$ and
let $Q_1\in[0,c_-]$ and $Q_2\in[c_+,1]$ be the points such that $\hat{f}(Q_1)=\hat{f}(1)$ and $\hat{f}(Q_2)=\hat{f}(0)$. By definition of $m$ we obtain the following  inequalities
$$
c_{m+1}\leq Q_1<c_i<Q_2\leq c'_{m+1},\quad\text{for}\ i=1,\ldots,m.
$$
Denote by $i_0$ the index $c_{i_0}=\min\{c_0,c_1,\ldots,c_m\}$ and observe that $c_-\not\in \hat{f}^{i}((c_{m+1},c_{i_0}))$ for $i=0,\ldots,m$.
Therefore $f^i((\pi(c_{m+1}),\pi(c_{i_0})))$ does not contain $c$ for $i=0,\ldots,m$.
By definition of $c_{m+1}$ also $f^{m+1}((\pi(c_{m+1}),\pi(c_{i_0})))$ does not contain $c$.
Therefore $f^i|_{(\pi(c_{m+1}),\pi(c_{i_0}))}$ is continuous (and monotone) for $i=0,\ldots, m+2$.
Furthermore
\begin{eqnarray*}
\hat{f}^{i_0}([c_{m+1},c_{i_0}])&=&[c_{m+1-i_0},c_0],\\
\hat{f}([c_{m+1-i_0},c_0])&=&[c_{m-i_0},1]\supset[c_{m-i_0},c'_{m+1}]
\end{eqnarray*}
and
\begin{eqnarray*}
\hat{f}^{m-i_0}([c_{m-i_0},c'_{m+1}])&=&[c_0,c_{i_0+1}]\supset[c_+,c_{i_0+1}],\\
\hat{f}([c_+,c_{i_0+1}])&=&[0,c_{i_0}]\supset[c_{m+1},c_{i_0}].
\end{eqnarray*}

This implies that
$$
[c_{m+1},c_{i_0}]\subset\hat{f}^{m+2}([c_{m+1},c_{i_0}]).
$$
But then $[\pi(c_{m+1}),\pi(c_{i_0})]\subset \overline{f^{m+2}((\pi(c_{m+1}),\pi(c_{i_0})))}$.
Extending $f^{m+2}|_{(\pi(c_{m+1}),\pi(c_{i_0}))}$ to a continuous map $\tilde{f}$ on $[\pi(c_{m+1}),\pi(c_{i_0})]$ we obtain a fixed point
$p\in [\pi(c_{m+1}),\pi(c_{i_0})]$ for $\tilde{f}$. If $p\in (\pi(c_{m+1}),\pi(c_{i_0}))$ then $p$ is also a fixed point of $f^{m+2}$
and therefore also $\hat{f}^{m+2}$. In the other case, by continuity of $\tilde f$, there is a sequence $\tilde{z}_j\in (\pi(c_{m+1}),\pi(c_{i_0}))$ such that $\lim_{j\to \infty}\tilde{z}_j=\lim_{j\to \infty}\tilde{f}(\tilde{z}_j)=p$.
Since points $\tilde{z}_j$ are pairwise distinct and $\lim_{j\to 
	\infty} |\tilde{f}(\tilde{z}_j)-\tilde{z}_j|=0$
taking $z_j\in \pi^{-1}(\tilde{z_j})\subset (c_{m+1},c_{i_0})$ we obtain that $\lim_{j\to 
	\infty} d(\hat{f}^{m+2}(z_j),z_j)=0$. But this implies that $\hat{f}^{m+2}$ has a fixed point in $[c_{m+1},c_{i_0}]$.

The case $\hat{f}^{-m}(c_+)\in[\hat{f}(0),\hat{f}(1)]$ is dealt similarly.
\end{proof}

We also present a version of Lemma~4.2 from \cite{Ding} with appropriately adjusted proof.

\begin{lem}\label{lem:32Dingimproved}
	Let $f$ be an expanding Lorenz map, $\hat{f}$ be the map induced on $\mathbb{X}$ and $1 < \kappa < \infty$ be the smallest period of the periodic points of $\hat{f}$. Suppose $\hat{O}$ is a $\kappa$-periodic orbit of $\hat{f}$,
	$\hat{P}_L =\max\{x\in\hat{O}:x\leq c_-\}$, $\hat{P}_R =\min\{x\in\hat{O}:x\geq c_+\}$. If $\hat{P}_L\neq c_-$ and $\hat{P}_R\neq c_+$, then we have 
	$$
	N((\pi(\hat{P}_L), c)) = N((c, \pi(\hat{P}_R))) = \kappa.
	$$
\end{lem}
\begin{proof}
Observe that $\hat{P}_L\neq c_-$ and $\hat{P}_R\neq c_+$ implies $\pi(\hat{P}_L)<c$ and $\pi(\hat{P}_R)>c$. So 
$$
\eta:=N((\pi(\hat{P}_L), c))\quad\text{and}\quad\xi:=N((c, \pi(\hat{P}_R)))
$$ 
are well-defined natural numbers. Suppose $\eta\neq\kappa$ and consider the following cases:

\textbf{Case 1.} $\eta>\kappa$. Since $\hat{f}^\kappa(\hat{P}_L)=\hat{P}_L$, we must have $\hat{f}^\kappa((\hat{P}_L,c_-))\subset(\hat{P}_L,c_-)$. Therefore we obtain a contradiction, because $f$ is expanding.

\textbf{Case 2.} $\eta<\kappa$. In this case there is a point $z\in(\hat{P}_L, c_-)$ such that $\pi(\hat{f}^\eta(z))=c$, i.e. $\hat{f}^\eta(z)=c_-$ or $\hat{f}^\eta(z)=c_+$. Hence, for the interval $[\hat{P}_L,z]$ we obtain
$$
\hat{f}^\eta([\hat{P}_L,z])=[\hat{f}^\eta(\hat{P}_L),\hat{f}^\eta(z)]\supset[\hat{P}_L,z],
$$
because $\hat{f}^\eta(\hat{P}_L)<\hat{P}_L$. 
Note that $f^\eta|_{(\pi(\hat{P}_L),\pi(z))}$ is monotone and
$[\pi(\hat{P}_L),\pi(z)]\subset (f^\eta(\pi(\hat{P}_L)),f^\eta(\pi(z)))$.
This implies by continuity of $f^\eta$ that there is $[a,b]\subset (\pi(\hat{P}_L),\pi(z))$
such that $f^\eta ([a,b])\supset [a,b]$.
So there exists a periodic point $p\in(\pi(\hat{P}_L),\pi(z))$ for $f$ with period $\eta<\kappa$. 
But then $\hat{f}$ also has a point of period $\eta$.
We get a contradiction with definition of $\kappa$.

Therefore $\eta=\kappa$. Similarly we prove that $\xi=\kappa$.
\end{proof}

\begin{thm}\label{cor:srong-trans}
Let $f$ be a transitive and expanding Lorenz map  and let $\hat{f}$  be the map induced on $\mathbb{X}$. 
Then $\hat{f}$ is strongly transitive, that is
for every nonempty open set $U\subset \mathbb{X}$ there is $k$ such that $\bigcup_{i=0}^k \hat{f}^{i}(U)=\mathbb{X}$.
\end{thm}
\begin{proof}
Fix any open set $U\subset \mathbb{X}$. Firstly, observe that
$$
\hat{f}^n(U)\setminus\left( \Orb(\hat{f}(0))\cup\Orb(\hat{f}(1))\right)\subset\Int\hat{f}^n(U)
$$
for each $n\in\mathbb{N}_{0}$. Hence, by Theorem~\ref{thm:TransDenseOrbits} we obtain
$$
\begin{aligned}
\mathbb{X}\setminus\left( \Orb(\hat{f}(0))\cup\Orb(\hat{f}(1))\right)&=\bigcup_{n=0}^\infty\left[  \hat{f}^n(U)\setminus\left( \Orb(\hat{f}(0))\cup\Orb(\hat{f}(1))\right)\right]\\
&\subset\bigcup_{n=0}^\infty\Int\hat{f}^n(U).
\end{aligned}
$$

\textbf{Case 1}.  Assume that the points $\hat{f}(0)$ and $\hat{f}(1)$ are not periodic. We are going to show that in this case the orbits of the points $\hat{f}(0)$ and $\hat{f}(1)$ are contained in $\bigcup_{n=0}^\infty\Int\hat{f}^n(U)$. Since $f(0)<f(1)$, there is a point $x_0\in(c_+,1)$ such that $\hat{f}(x_0)=\hat{f}(0)$. Similarly, there is a point $x_1\in(0,c_-)$ such that $\hat{f}(x_1)=\hat{f}(1)$. 

We claim that at least one of the points $x_0$ and $x_1$ does not belong to the set $\Orb(\hat{f}(0))\cup\Orb(\hat{f}(1))$. Indeed, suppose $x_0$, $x_1\in\Orb(\hat{f}(0))\cup\Orb(\hat{f}(1))$. We must have $x_0\in\Orb(\hat{f}(1))$ and $x_1\in\Orb(\hat{f}(0))$. But then $\hat{f}^{i}(1)=x_0$, $\hat{f}^{j}(0)=x_1$ and therefore $\hat{f}^{i+j+1}(1)=\hat{f}(1)$ for some $i,j\geq 1$, which is a contradiction with our assumption, proving the claim. 

Let $x_0\notin\Orb(\hat{f}(0))\cup\Orb(\hat{f}(1))$. Then $x_0\in\bigcup_{n=0}^\infty\Int\hat{f}^n(U)$ which implies $\Orb(x_0)\subset\bigcup_{n=0}^\infty\Int\hat{f}^n(U)$, because $0$, $1\notin\Orb(x_0)$. Moreover, either $x_1\notin\Orb(\hat{f}(0))\cup\Orb(\hat{f}(1))$ or $x_1\in\Orb(\hat{f}(0))\subset \Orb(x_0)$.
In the second case it is clear that $\Orb(x_1)\subset\bigcup_{n=0}^\infty\Int\hat{f}^n(U)$.
In the first case it is clear that $1\not\in \Orb(x_1)$ because $\hat{f}(1)$ is not periodic, and also $0\not\in \Orb(x_1)$ because by definition $c_+\not\in \Orb(1)$.
Therefore, repeating argument applied before to  $x_0$, we see that also in this case $\Orb(x_1)\subset\bigcup_{n=0}^\infty\Int\hat{f}^n(U)$. Therefore $\bigcup_{n=0}^\infty\Int\hat{f}^n(U)=\mathbb{X}$.

\textbf{Case 2}. Assume that the points $0$ and $1$ are not periodic, but at least one of $\hat{f}(0)$ and $\hat{f}(1)$ is. Let $\hat{O}$ be the minimal cycle of $\hat{f}$. Clearly $c_-$, $c_+\notin \hat{O}$, hence the points 
$$
\hat{P}_L:=\max\{x\in \hat{O}: x<c_-\}\quad\text{and}\quad\hat{P}_R:=\min\{x\in \hat{O}: x>c_+\}
$$
are well-defined. Moreover, by Theorem~\ref{thm:TransDenseOrbits} we have $\hat{P}_L$, $\hat{P}_R\in\bigcup_{i=0}^\infty\hat{f}^i(U)$. First, we will prove that there are $l$, $r\in\mathbb{N}_0$ such that the set $\hat{f}^l(U)$ contains a left-sided neighborhood of $\hat{P}_L$ and the set $\hat{f}^r(U)$ contains a right-sided neighborhood of $\hat{P}_L$, i.e. $(a,\hat{P}_L]\subset\hat{f}^l(U)$ and $[\hat{P}_L,b)\subset\hat{f}^r(U)$ for some points $a<\hat{P}_L<b$.  Suppose it is not the case.
In particular $\hat{P}_L\not\in\bigcup_{n=0}^\infty \Int\hat{f}^n(U)$, hence  $\hat{P}_L\in \Orb(\hat{f}(0))\cup\Orb(\hat{f}(1))$.
Assume that there is a point $b>\hat{P}_L$ such that $[\hat{P}_L,b)\subset\hat{f}^r(U)$ for some $r\in\mathbb{N}_0$, but for any $i\in\mathbb{N}_0$ and $a<\hat{P}_L$ there is a point $x\in(a,\hat{P}_L)$ such that $x\notin\hat{f}^i(U)$. 
Note that in the case $\hat{P}_L\not\in \orb(\hat{f}(0))$, if we pick $z_0\in U$ such that $\hat{f}^i(z_0)=\hat{P}_L$ then for some $e$ sufficiently close to $z_0$ we have
$(e,z_0]\subset U$ and $(\hat{f}^i(e),\hat{P}_L]=(\hat{f}^i(e),\hat{f}^i(z_0)]\subset \hat{f}^i(U)$ which is a contradiction, therefore $\hat{P}_L\in\orb(\hat{f}(0))$. We claim that $\hat{f}(0)$ is periodic. 
If we assume the contrary then
there is a point $x_0\in(c_+,1)$ such that $\hat{f}(x_0)=\hat{f}(0)$ and $x_0\not\in \orb(\hat{f}(0))$. Then $x_0\in \orb(\hat{f}(1))$ because otherwise $x_0\in \Int \hat{f}^j(U)$
for some $j$, and since we are not passing through any of $c_+,c_-$, this gives $\hat{P}_L\in\Int\hat{f}^n(U)$ for some $n$ which is a contradiction. But $x_0\in \orb(\hat{f}(1))$ is a contradiction as well, since $\hat{f}(1)$ is periodic while $x_0$ cannot.  Indeed, the claim holds and $\hat{f}(0)$ is periodic, which gives $\orb(\hat{f}(0))=\hat{O}$.
The symmetric case $(a,\hat{P}_L]\subset\hat{f}^l(U)$ leads to $\orb(\hat{f}(1))=\hat{O}$.

Let us assume that the first case holds, that is $\orb(\hat{f}(0))=\hat{O}$ and $\hat{f}^{-1}(\hat{O})=\hat{O}\cup\{0\}$.
Since $c_-$, $c_+\notin \hat{O}$, we can order the points from the orbit $\hat{O}$ as follows
\begin{equation}
\label{eq:nkcycle1}
\hat{z}_0<\hat{z}_1<\ldots<\hat{z}_{n-k-1}<c_-<c_+<\hat{z}_{n-k}<\ldots<\hat{z}_{n-1}
\end{equation}
for some $k\in\{1,2,\ldots,n-1\}$. Now, we are going to show that the orbit $O=\pi(\hat{O})$ forms a primary $n(k)$-cycle for $f$. 
Denote $z_i=\pi(\hat{z}_i)$ for $i=0,1,\ldots,n-1$. 
By monotonicity $\hat{f}(\hat{z}_{n-k+j})=\hat{z}_{j}$ for $j=0,1,\ldots,k-1$ as otherwise 
$$
Z=\{\hat{z}_0,\hat{z}_1,\ldots, \hat{z}_{k-1}\}\setminus \{\hat{f}(\hat{z}_{n-k}),\hat{f}(\hat{z}_{n-k+1}),\ldots, \hat{f}(\hat{z}_{n-1})\}\neq \emptyset.
$$
Clearly, we must have $\hat{z}_0=\hat{f}(\hat{z}_{n-k})$ and let $y=\min Z$. Then $f(\hat{z}_0)=y$
and there is also $i>n-k$ such that $\hat{f}(\hat{z}_{i})$ is to the right of $y$.
Then $y$ has a preimage in $(c_+,0)$ that does not belong to $\hat{O}$,
which is a contradiction with $\hat{f}^{-1}(\hat{O})=\hat{O}\cup\{0\}$.
By induction, it extends to
$\hat{f}(\hat{z}_i)=\hat{z}_{i+k}$ for $i=0,1,\ldots,n-k-1$, so $f(z_i)=z_{i+k (\text{mod }n)}$ for all $i=0,1,\ldots,n-1$. 
By definition $0<\hat{z}_0$, thus $\hat{f}(0)<\hat{z}_k$.
If $\hat{f}(0)<\hat{z}_{k-1}$ then there is a point $x\in(0,\hat{z}_0)$ such that $\hat{f}(x)=\hat{z}_{k-1}$ which is a contradiction with $\hat{f}^{-1}(\hat{O})=\hat{O}\cup\{0\}$. So $\hat{f}(0)\geq\hat{z}_{k-1}$ and therefore $\hat{f}(0)=\hat{z}_{k-1}$. By the same argument applied to $\hat{f}(1)$ we obtain that $\hat{f}(1)\in(\hat{z}_{k-1},\hat{z}_k]$. 
Using the assumption about $\hat{f}^{-1}(\hat{O})$ for the last time, we obtain $1\notin\hat{f}^{-1}(\hat{O})$ and so $\hat{f}(1)\in(\hat{z}_{k-1},\hat{z}_k)$. Then the points
$$
z_0<z_1<\ldots<z_{n-k-1}<c<z_{n-k}<\ldots<z_{n-1}
$$
form primary $n(k)$-cycle for $f$ such that $f(1)<z_k$. This by \cite[Theorem 6.5]{OPR} implies that $f$ is not transitive, which is a contradiction.
Therefore $(a,\hat{P}_L]\subset\hat{f}^l(U)$, $[\hat{P}_L,b)\subset\hat{f}^r(U)$ for some $l$, $r\in\mathbb{N}_0$. 
The proof for symmetric case $\hat{f}^{-1}(\hat{O})=\hat{O}\cup\{1\}$ with $\orb(\hat{f}(1))=\hat{O}$ is analogous.

Since $[\hat{P}_L,b)\subset \hat{f}^r(U)$ for some $r$ and $f$ is expanding, there is $i\geq 0$
such that $[\hat{P}_L,c_+]\subset\hat{f}^{r+i}(U)$. Similarly, there is the smallest iteration $s\geq 0$
such that $c_-\in (\hat{f}^s(a),\hat{f}^s(\hat{P}_L)]$ and therefore $[c_-,\hat{P}_R]\subset\hat{f}^{l+s}(U)$.
	It follows that there is $m\in\mathbb{N}_0$ such that
	\begin{equation}\label{PLPR=covered}
	[\hat{P}_L,\hat{P}_R]\subset\bigcup_{i=0}^{m}\hat{f}^i(U).
	\end{equation}

Now we will prove that
$$
\bigcup_{i=0}^{n-1}\hat{f}^i([\hat{P}_L,\hat{P}_R]) =\mathbb{X},
$$
where $n=|\hat{O}|$. Observe that
$$
\hat{f}([\hat{P}_L,c_-])=[\hat{f}(\hat{P}_L),1]\subset (c_+,1],\quad \hat{f}^2([\hat{P}_L,c_-])=[\hat{f}^2(\hat{P}_L),\hat{f}(1)]
$$
and
$$
\hat{f}([c_+,\hat{P}_R])=[0,\hat{f}(\hat{P}_R)]\subset [0,c_-),\quad \hat{f}^2([c_+,\hat{P}_R])=[\hat{f}(0),\hat{f}^2(\hat{P}_R)].
$$
	By Lemma~\ref{lem:32Dingimproved} we know that	
\begin{equation}
\label{eq:cnn}
N((\pi(\hat{P}_L), c)) = N((c, \pi(\hat{P}_R)))=n.
\end{equation}
It means that $c_-$, $c_+\notin\hat{f}^i((\hat{P}_L,c_-))$ and $c_-$, $c_+\notin\hat{f}^i((c_+,\hat{P}_R))$ for $i=0,1,\ldots,n-1$. Furthermore $c_-$, $c_+\notin\hat{f}^i([\hat{P}_L,\hat{P}_R])$ for $i=1,\ldots,n-1$, because $c_-$, $c_+$ are not periodic. 

We also get that there are $c_n\in(\hat{P}_L,c_-)$ and $c'_n\in(c_+,\hat{P}_R)$ such that $\pi(\hat{f}^n(c_n))=\pi(\hat{f}^n(c'_n))=c$. Hence 
\begin{equation}
\label{pre-z}
c_{n-2}:=\hat{f}^2(c_n)\in(\hat{f}^2(\hat{P}_L),\hat{f}(1)),\quad c'_{n-2}:=\hat{f}^2(c'_n)\in(\hat{f}(0),\hat{f}^2(\hat{P}_R))
\end{equation}
and $\pi(\hat{f}^{n-2}(c_{n-2}))=\pi(\hat{f}^{n-2}(c'_{n-2}))=c$. Suppose that $\pi(c_{n-2})\neq \pi(c'_{n-2})$. 
Then $|\hat{f}^{-(n-2)}(c_-)|>1$ and $|\hat{f}^{-(n-2)}(c_+)|>1$ by \eqref{pre-z}, 
which implies that
$$
m=\min\{i\in\mathbb{N}_0:\hat{f}^{-i}\left(\{c_-,c_+\} \right) \cap[\hat{f}(0),\hat{f}(1)]\neq\emptyset\}\leq n-3
$$
and by Lemma~\ref{lem:41Dingimproved} there exists a periodic point with period $m+2<n$. Hence we obtain a contradiction, because $n$ is the smallest period of the periodic points of $\hat{f}$. 

So $\pi(c_{n-2})= \pi(c'_{n-2})$ and therefore by \eqref{pre-z} we have $(\hat{f}^2(\hat{P}_L),\hat{f}(1))\cap(\hat{f}(0),\hat{f}^2(\hat{P}_R))\neq\emptyset$ which in particular implies that $\hat{f}^2(\hat{P}_L)<\hat{f}^2(\hat{P}_R)$. Since $\hat{f}(0)<\hat{f}(1)$ and $c_-$, $c_+\notin\hat{f}^i([\hat{P}_L,\hat{P}_R])$ for $i=1,\ldots,n-1$, we have
$$
[\hat{f}^2(\hat{P}_L),\hat{f}^2(\hat{P}_R)]\subset\hat{f}^2([\hat{P}_L,\hat{P}_R])
$$
and
\begin{eqnarray*}
[\hat{f}^{j+2}(\hat{P}_L),\hat{f}^{j+2}(\hat{P}_R)]&=&\hat{f}^j([\hat{f}^2(\hat{P}_L),\hat{f}^2(\hat{P}_R)])\\
&\subset&\hat{f}^{j}(\hat{f}^2([\hat{P}_L,\hat{P}_R]))
=\hat{f}^{j+2}([\hat{P}_L,\hat{P}_R])
\end{eqnarray*}
for $j=0,1,\ldots,n-2$. Therefore
$$
\bigcup_{i=0}^{n-1}\hat{f}^i([\hat{P}_L,\hat{P}_R]) =\mathbb{X}
$$
which by \eqref{PLPR=covered} for some $m$ gives
$$
\mathbb{X}=\bigcup_{i=0}^{n-1}\hat{f}^i([\hat{P}_L,\hat{P}_R])\subset\bigcup_{i=0}^{m+n-1}\hat{f}^i(U).
$$

\textbf{Case 3}.  As the last possibility, assume that $0$ is a periodic point of period $s$(the case of $1$ periodic is dealt the same way). There is $x\in \pi^{-1}(\Int \pi(U))$ and a minimal $n$ such that $\hat{f}^n(x)=c_+$.
Then $0\in \Int \hat{f}^{n+1}(U)$. This covers the case when $0$ is fixed point, so let us assume that $0$ is not a fixed point.
If we take small neighborhood of $0$, say $[0,\eps)\subset \Int \hat{f}^{n+1}(U)$ then $\hat{f}^i([0,\eps))$ is an open set in $\mathbb{X}$ for $i=0,1,\ldots,s-1$
	because $\hat{f}^i(0)$ is the right endpoint of a hole in $\mathbb{X}$.
Therefore $\orb(0)\subset \bigcup_{n=0}^\infty\Int\hat{f}^n(U).$ If $1$ is also periodic then the proof is completed by a symmetric argument, so assume that $1$ is not periodic.

Assume that $\hat f(1)$ is not periodic.
Since $f(0)<f(1)$ there is 
$x_1\in(0,c_-)$ such that $\hat f(x_1)=\hat f(1)$.
Then $x_1\in \Int \hat{f}^m(U)$ for some $m$, and since $\orb(\hat f(1))$ does not pass through $c_-,c_+$ we have
$\hat f^i (1)\in \Int \hat{f}^{m+i}(U)$. This shows 
$$
\mathbb{X}\subset\bigcup_{n=0}^\infty\Int\hat{f}^n(U).
$$
So to complete the proof, let us assume that $\hat{f}(1)$ is periodic. Denote $A=\Orb(\hat{f}(1))$ and observe that since $A$ is not passing through $c_-,c_+$, if
$\hat{f}^{-1}(A)\neq A\cup\{1\}$ then $A\subset \bigcup_{n=0}^\infty\Int\hat{f}^n(U)$ and the proof is completed.
But if $\hat{f}^{-1}(A)= A\cup\{1\}$ then we can repeat discussion after \eqref{eq:nkcycle1} obtaining that $\pi(A)$
is a primary $n(k)$-cycle such that $z_{k-1}<f(0)<f(1)=z_k$, contradicting transitivity. 
The proof is completed.
\end{proof}

\begin{rem}
Note that Theorem~\ref{thm:TransDenseOrbits} could be deduced immediately from the statement of Theorem~\ref{cor:srong-trans}. The problem is, that
Theorem~\ref{thm:TransDenseOrbits} is one of main tools used in its proof.
\end{rem}

\begin{thm}\label{thm:28}
	If $f$ is an expanding Lorenz map with renormalization $g$ which additionally is locally eventually onto, then $F_g=[0,1]$.
\end{thm}
\begin{proof}
	First we claim that there is $\eta>0$ such that $f^{\eta}(c_-)=f^{\eta}(c_+)$.
	Since $f$ is locally eventually onto with renormalization, by Corollary~6.3 in \cite{OPR}
	there is an open interval $J$ and $n>0$ such that $J$ is homeomorphic to $(0,c)$ or $(c,1)$ under $f^{n}$ but one of the maps $f^j|_{J}$ is not continuous for some $0<j<n$.
	We may also assume that $j$ is the minimal such number.
	In particular, there is $y\in J$ such that $f^{j-1}(y)=c$. But if we denote $z=f^n(y)$, then since $f^n|_J$ is continuous we have $f^n(y_+)=f^n(y_-)=z$.
	Let $\eta=n-j+1>0$. Then 
	\begin{equation}
	f^\eta(c_+)=f^n(y_+)=z=f^n(y_-)=f^\eta(c_-)\label{eq:matchmain}
	\end{equation}	
	and so the claim holds.
	
	Next assume that renormalization $g = (f^m, f^k)$ is defined on interval $[u,v]$.
	We will prove that $\hat{F}_g$ is completely invariant. Assume on the contrary that it is not the case, which by Lemma~\ref{lem:inv}
	means that $\orb(u)\cap (u,v)=\emptyset$ or $\orb(v)\cap (u,v)= \emptyset$. 
	Assume without loss of generality that the first case holds. 
	Note that $g(c_+)=u$ and $g(c_-)=v$
	and so there are $i,j>0$ such that 
	\begin{equation}
	\label{eq:uvjz}
	f^i(u)=f^j(v_-)=z.
	\end{equation}	
	But $f^s$ is continuous at $u$ for every $s>0$ and  
	$g([u,v])\subset [u,v]$ which implies $u$ is periodic.  
	Assume that $\orb(v)\cap (u,v)\neq \emptyset$ and observe that by \eqref{eq:uvjz} the number
	$r=\max\{t>0 : f^t(v_-)\in (u,v)\}$ satisfies $r\leq j$. 
	Denote $w=f^r(v_-)$ and observe that $w= c$, because otherwise $g((u,c) \cup (c,v))\subset (u,v)$ contradicts the definition of $r$. 
But then $v_-$ is periodic and so is the trajectory of $c_-$ 
which contradicts \eqref{eq:uvjz} because orbit of $z$ never visits $(u,v)$.
	We obtain that $\orb(v)\cap (u,v)=\emptyset$ and therefore both points are periodic. By definition of $\eta$ they belong to the same periodic orbit.
	This is in contradiction with definition of $\eta$, because if we take $l>\eta$ such that $u,v$ are fixed for $f^l$ then
	$$
	u=f^l(u)=f^{l-\eta}(f^{\eta}(u))=f^{l-\eta}(f^{\eta}(v))=f^l(v)=v.
	$$
A contradiction again.
	Indeed $\hat{F}_g$ is completely invariant. 
	
Suppose that $\hat{F}_g\neq\mathbb{X}$, which in particular means that $\hat{J}_g=\mathbb{X}\setminus\hat{F}_g\neq\emptyset$ is completely invariant. By Remarks \ref{rem:cominv} and \ref{rem:FgJg} it follows that $J_g$ is a proper completely invariant closed subset of $[0,1]$.
By Lemma~\ref{lem:inv}(\ref{lem:inv:2.3}) the map $f$ has no fixed points, which is in contradiction with Corollary~\ref{cor:TransInv}. 
Therefore $\hat{F}_g=\mathbb{X}$ and by surjectivity of the projection $\pi$ we obtain
		$$
		[0,1]=\pi(\mathbb{X})=\pi(\pi^{-1}(F_g))=F_g.
		$$
The proof is completed.
\end{proof}

The property that there is some $\eta\in\mathbb{N}$ such that $f^\eta(c_-)=f^\eta(c_+)$ is called \textit{matching} in the literature and it is a subject of current research (see for instance \cite{Bru1}, \cite{Bru2}). Combining the first part of the proof of Theorem~\ref{thm:28} with the above terminology gives the following.

\begin{cor}\label{cor:mathcing}
	If $f$ is an expanding Lorenz map with renormalization $g$ which additionally is locally eventually onto, then $f$ has matching.
\end{cor}

In the \cite{Bru1}, the authors raise the question about parameters $\beta $ for which generalized $\beta$-transformation has matching (this class contain all linear Lorenz maps). They ask whether there exist  generalized $\beta$-transformations with non-Pisot and non-Salem number $\beta$ and matching (see end of p.72 in \cite{Bru1}). Our Example \ref{51from4} answers this question affirmatively, showing that such map exists, since map constructed there has matching by Corollary~\ref{cor:mathcing}.

	\begin{rem}
Let $\hat{f}$ be a map induced on $\mathbb{X}$ by an expanding Lorenz map $f$ with primary $n(k)$-cycle
	$$
	\hat{z}_0<\hat{z}_1<\dots<\hat{z}_{n-k-1}<c_-<c_+<\hat{z}_{n-k}<\dots<\hat{z}_{n-1}.
	$$
It may happen that exactly one of the following conditions holds
		$$
		\hat{z}_{k-1}=\hat{f}(0)\quad\text{or}\quad\hat{z}_{k}=\hat{f}(1)
		$$
		as presented in next example, $D_O=\mathbb{X}$ and the map is still not transitive. In fact there are numerous examples of this kind.
		This shows that implication in Corollary~\ref{cor:transitive} cannot be easily reversed.
	\end{rem}
\begin{exmp}
Consider an expanding Lorenz map $f\colon [0,1]\to[0,1]$ defined by $f(x)=\beta x+\alpha (\text{mod }1),$ 
where
	$$
	\beta:=\sqrt[3]{2},\quad\alpha:=\frac{1}{\beta+\beta^2}=\frac{1}{\sqrt[3]{2}+\sqrt[3]{4}}\approx0.351207.
	$$
Let $\hat{f}$ denote the map induced by $f$ on $\mathbb{X}$. Denote 
$$
z_0:=f(0)=\frac{1}{\sqrt[3]{2}+\sqrt[3]{4}}\quad\text{ and }\quad z_1:=f^2(0)=\frac{1}{\sqrt[3]{2}}.
$$
Then $O=\{z_0,z_1\}$ forms a primary $2(1)$-cycle for $f$. Observe that $f(1)<z_1$, so by \cite[Theorem 6.5]{OPR}
the map $\hat{f}$ is not transitive. We may also view $z_0,z_1\in \mathbb{X}$ and denote $\hat{O}=\{z_0,z_1\}$.
Note that $c_+\in \hat{f}^{-2}(z_0)$ and therefore by Remark~\ref{rem:density} we see that
$$
D_O=\overline{\bigcup_{i=0}^\infty \hat{f}^{-i}(\hat{O})}=\mathbb{X}.
$$
\end{exmp}

\section{Results related to Theorem~\ref{thm:CuiMain}}\label{sec:periodic}
Recall that the rotation number of $x\in\mathbb{X}$ under $\hat{f}$ is defined as
$$
\rho(x):=\lim\limits_{n\to\infty}\frac{m_n(x)}{n},
$$
provided that this limit exists, where
$$
m_n(x):=\left| \{i\in\{0,1,\dots,n-1\}:\hat{f}^i(x)\geq c_+\}\right|.
$$
The set
$$
P(\hat{f}):=\{\rho\in\mathbb{R}:\rho(x)=\rho\text{ for some }x\in\mathbb{X}\}
$$
is called the rotation interval of $\hat{f}$.

\begin{thm}
	\label{thm:RotInt}
	Let $f$ be an expanding Lorenz map with a primary $n(k)$-cycle
	$$
	z_0<z_1<\dots<z_{n-k-1}<c<z_{n-k}<\dots<z_{n-1}.
	$$
	Then the rotation interval degenerates to singleton $\{\frac{k}{n}\}$. Moreover, $f$ is periodic renormalizable if and only if $z_{k-1}\neq f(0)$ and $z_k\neq f(1)$.
\end{thm}
\begin{proof}
The proof that the rotation number is unique follows similar argument as in \cite[Proposition~3]{Cui} for $2$-periodic orbit.
Indeed, note first that by definition of $n(k)$-cycle we have $|\pi^{-1}(z_i)|=|\{\hat{z}_i\}|=1$ for $i=0,1,\dots,n-1$. Furthermore $\hat{z}_0>0$ and $\hat{z}_{n-1}<1$ because, if $\hat{z}_0=0$ or $\hat{z}_{n-1}=1$ then $\hat{f}(0)=\hat{z}_k\geq \hat{f}(1)$ or $\hat{f}(1)=\hat{z}_{k-1}\leq \hat{f}(0)$ which contradicts the fact that $f$ is expanding. 
	Hence, if we denote $\hat{O}:=\{\hat{z}_i:i=0,1,\dots,n-1\}$  then there is a partition of $\mathbb{X}\setminus \hat{O}$ consisting of $n+2$ intervals
	\begin{align*}
		I_{-1}&:=[0,\hat{z}_0),\hspace{0.2cm}I_0:=(\hat{z}_0,\hat{z}_1),\hspace{0.2cm}\ldots,\hspace{0.2cm}I_{n-k-1}:=(\hat{z}_{n-k-1},c_-],\\
		I_{n-k}&:=[c_+,\hat{z}_{n-k}),\hspace{0.2cm}\ldots,\hspace{0.2cm}I_{n-1}:=(\hat{z}_{n-2},\hat{z}_{n-1}),\hspace{0.2cm}I_n:=(\hat{z}_{n-1},1].
	\end{align*}
	For $z\in\mathbb{X}$ and $j\in\mathbb{N}_0$ we will denote by $Orb_j(z)$ the first $j$ elements of the orbit of point $z$ under $\hat{f}$. Clearly, the rotation number of any point from the $n(k)$-cycle $\hat{O}$ is equal to $\frac{k}{n}$ and the same is true for any point that is eventually mapped to a point in $\hat{O}$. Let $x\in\mathbb{X}$ and assume that the orbit of $x$ is disjoint with the $n(k)$-cycle, i.e.  $\orb(x)\cap \hat{O}=\emptyset$. Then, if $x\in I_{n-k-1}$ we have 
	$$
	Orb_n(x)\cap I_i\neq\emptyset\quad\text{for}\quad i=0,1,\ldots,n-k-1,n-k+1,\ldots,n-1,n,
	$$
	Similarly, if $x\in I_{n-k}$ we have 
	$$
	Orb_n(x)\cap I_i\neq\emptyset\quad\text{for}\quad i=-1,0,1,\ldots,n-k-2,n-k,\ldots,n-1.
	$$
	In both cases $\hat{f}^n(x)\in I_{n-k-1}\cup I_{n-k}$. It follows that $m_{l\cdot n}(x)=l\cdot k$ for any $l\in\mathbb{N}$. Observe that if $x$ belongs to any interval $I_i$ then there is $j\in\mathbb{N}_0$ such that $Orb_j(x)\cap I_{n-k-1}\neq\emptyset$ or $Orb_j(x)\cap I_{n-k}\neq\emptyset$. Assume that $j$ is the smallest such number for $x$. Then
	\begin{align*}
		\rho(x)&=\lim\limits_{i\to\infty}\frac{m_i(x)}{i}=\lim\limits_{i\to\infty}\frac{m_{j}(x)+m_{i-j}(\hat{f}^j(x))}{j+(i-j)}=\lim\limits_{i\to\infty}\frac{m_{i-j}(\hat{f}^j(x))}{i-j}\\
		&=\lim\limits_{p\to\infty}\frac{m_{p}(\hat{f}^j(x))}{p}=\frac{k}{n}.
	\end{align*}

	Next, note that by the above partition, $n$ is the smallest period of the periodic points of $\hat{f}$ and by Theorem~\ref{thm:Bimprov} the $n(k)$-cycle $\hat{O}$ is the unique $n$-periodic orbit. Consider the set 
	$$
	\hat{D}:=\overline{\bigcup_{i=0}^\infty \hat{f}^{-i}(\hat{O})}.
	$$
If $z_{k-1}\neq f(0)$ and $z_{k}\neq f(1)$ then
	$$
	\hat{f}^{-1}(\hat{O})\cap\bigcup_{i=-1}^{n}I_i=\emptyset
	$$
	and so $\hat{f}^{-1}(\hat{O})\subset \hat{O}$ which gives $\hat{D}=\hat{O}$. Clearly $\hat{f}(\hat{O})=\hat{O}$, so the set $\hat{O}$ is completely invariant. Next, observe that $\pi^{-1}(O)=\hat{O}$ and
	$$
	D=\overline{\bigcup_{i=0}^\infty f^{-i}(O)}=O,
	$$
	where $O=\{z_i:i=0,1,\dots,n-1\}$. So by Remark~\ref{rem:cominv} and Theorem~\ref{thm:A}(1) the renormalization $R_Df$ is well defined and periodic.
	
	On the other hand, if $z_{k-1}= f(0)$ or $z_{k}=f(1)$, then $0\in D$ or $1\in D$. Hence $D\neq O$ because $z_0>0$ and $z_{n-1}<1$. 	The proof is complete.	
\end{proof}

\begin{cor}\label{cor:periodic}
	Let $f$ be an expanding Lorenz map with a primary $n(k)$-cycle
	$$
	z_0<z_1<\dots<z_{n-k-1}<c<z_{n-k}<\dots<z_{n-1}
	$$
	such that $z_{k-1}\neq f(0)$ and $z_k\neq f(1)$. Then the renormalization $R_Df$ of $f$ associated to minimal cycle $O$ of $f$ is well defined, periodic and equal to $g=(f^n,f^n)$.
\end{cor}
\begin{proof}
By Theorem~\ref{thm:RotInt} the renormalization $R_Df$ is well defined and periodic, i.e.
$$
D=\overline{\bigcup_{i=0}^\infty f^{-i}(O)}=O,
$$
where $O=\{z_i:i=0,1,\dots,n-1\}$. It remains to show that $R_D f=(f^n,f^n)$. By Theorem~\ref{thm:cycle}\eqref{tw-n(k)-cykle:3.3}, it is enough to show that $J_g=O$, and by \eqref{tw-n(k)-cykle:3.2} in this theorem we have $O\subset J_g$.
Now, let us assume that there is a point $x\in J_g\setminus O$. Then, for some iterate we must have $f^j(x)\in(z_{n-k-1},z_{n-k})$. Since the set $J_g$ is completely invariant, $f^j(x)\in J_g$ and we get the contradiction with Theorem~\ref{thm:cycle}\eqref{tw-n(k)-cykle:3.2}. Indeed $J_g=O$.
\end{proof}

\section{Basic renormalizations}\label{sec:basic_renormalizations}

Analyzing the examples we have presented previously, one can notice that Theorem~\ref{thm:A}\eqref{thm:A:2} fails in the case of renormalizations $g=(f^l,f^r)$ with different numbers $l$ and $r$. In our examples, these are exactly the renormalizations that can be presented as a composition of renormalizations $f_i=(f_{i-1},f_{i-1}^2)$, $f_i=(f_{i-1}^2,f_{i-1})$ or $f_i=(f_{i-1}^2,f_{i-1}^2)$, where the last component $f_n=g$ is a trivial or special trivial renormalization of $f_{n-1}$. 
	
The above observations motivate the next theorem, which provides an algorithm to detect the renormalizations that can be successfully recovered from the corresponding sets $J_g$ by the procedure presented in Theorem~\ref{thm:A}. Let us note that the proper subsets $J_g\subset[0,1]$ obtained by applying the algorithm are not necessarily completely invariant.

\begin{thm}\label{thm:renorm_comp}
	Let $f\colon[0,1]\to[0,1]$ be an expanding Lorenz map and $n\in\N$. Assume that there is a sequence of Lorenz maps $\{f_i\}_{i=1}^n$ such that for every $i=1,2,\ldots,n$ the map $f_i\colon[u_i,v_i]\to[u_i,v_i]$ has one of the following forms:
	\begin{equation}\label{eq:basic_renorm}
	f_i=(f_{i-1},f_{i-1}^2),\quad f_i=(f_{i-1}^2,f_{i-1})\quad\text{or}\quad f_i=(f_{i-1}^2,f_{i-1}^2), 
	\end{equation}
	where $f_0:=f$. Then for every $i=1,2,\ldots,n$ there exist numbers $l_i$, $r_i\in\N$ such that $f_i=(f_0^{l_i},f_0^{r_i})$. Moreover, the following trichotomy holds.
	\begin{enumerate}
		\item\label{thm:renorm_comp:1} The map $f_n$ is the composition of $n$ trivial renormalizations, i.e. $f_i\neq(f_{i-1}^2,f_{i-1}^2)$ for every $i=1,2,\ldots,n$. Then $J_{f_n}\subset\{0,1\}$ and $l_n\neq r_n$.
		\item\label{thm:renorm_comp:2} We have $f_n=(f_{n-1}^2,f_{n-1}^2)$. Then the set $J_{f_n}$ contains a periodic orbit, $R_{J_{f_n}}f=f_n$ and $l_n=r_n$.
		\item\label{thm:renorm_comp:3} We have $f_n\neq(f_{n-1}^2,f_{n-1}^2)$ and $f_i=(f_{i-1}^2,f_{i-1}^2)$ for some $i<n$. Then the set $J_{f_n}$ contains a periodic orbit, but $R_{J_{f_n}}f\neq f_n$ and $l_n\neq r_n$. In fact $R_{J_{f_n}}f= f_j$, where $0<j<n$ is the largest index for which the map $f_j$ is neither trivial nor special trivial renormalization of $f_{j-1}$.
	\end{enumerate}
	In the above conditions we do not assume that maps $f_i$ are proper renormalizations, i.e. it may happen that $[u_i,v_i]=[u_{i-1},v_{i-1}]$ and $f_{i}$ is special trivial renormalization.
\end{thm}
\begin{proof}
	At first observe that by the assumption we have $f_1=(f_0^{l_1},f_0^{r_1})$, where
	$$
	(l_1,r_1)=(1,2)\quad\text{or}\quad(l_1,r_1)=(2,1)\quad\text{or}\quad(l_1,r_1)=(2,2).
	$$
	Assume that for some $k\in\{1,2\ldots,n-1\}$ there are numbers $l_k$, $r_k\in\N$ such that $f_k=(f_0^{l_k},f_0^{r_k})$. Let us consider the following three cases.
	\begin{itemize}
		\item [a)] If $f_{k+1}=(f_{k},f_{k}^2)$, then $f_{k+1}=(f_0^{l_k},f_0^{r_k+l_k})$ and we put $(l_{k+1},r_{k+1}):=(l_k,r_k+l_k)$.
		\item [b)] If $f_{k+1}=(f_{k}^2,f_{k})$, then $f_{k+1}=(f_0^{l_k+r_k},f_0^{r_k})$ and we put $(l_{k+1},r_{k+1}):=(l_k+r_k,r_k)$.
		\item [c)] If $f_{k+1}=(f_{k}^2,f_{k}^2)$, then $f_{k+1}=(f_0^{l_k+r_k},f_0^{r_k+l_k})$ and we put $(l_{k+1},r_{k+1}):=(l_k+r_k,l_k+r_k)$.
	\end{itemize}
The induction concludes the first part of the theorem. It remains to show that the trichotomy \eqref{thm:renorm_comp:1}--\eqref{thm:renorm_comp:3} holds. It is clear that
	these conditions cover all possible options, so we only need to verify statements about sets $J_{f_n}$.

	\textbf{(\ref{thm:renorm_comp:1}):} Since all renormalizations are trivial, for each $i=1,2,\ldots,n$ the map $f_i$ is defined on the interval $[u_i,v_i]=[f_{i-1}(u_{i-1}),v_{i-1}]$ or $[u_i,v_i]=[u_{i-1},f_{i-1}(v_{i-1})]$, where $[u_0,v_0]:=[0,1]$. In particular it implies, that the following inclusions hold:
	$$
	[0,1]=[u_0,v_0]\supset[u_1,v_1]\supset\ldots\supset[u_{n-1},v_{n-1}]\supset[u_n,v_n].
	$$
		
	Let us fix any number $k\in\{1,2,\ldots,n\}$. Note that for every point $x\in(u_{k-1},v_{k-1})$ we have $\Orb_{f_{k-1}}(x)\cap(u_k,v_k)\neq\emptyset$ and for $x\in\{u_{k-1},v_{k-1}\}$ we get the following equivalence
	$$
	\Orb_{f_{k-1}}(x)\cap(u_k,v_k)=\emptyset\iff f_{k-1}(x)=x.
	$$
	Suppose that $f_{k-1}(v_{k-1})=v_{k-1}$. We will show that $c\in\Orb_{f_0}(v_{k-1})$ or $v_{k-1}=1$. If $f_0(v_{k-1})=v_{k-1}$, then by the topological expanding condition for $f_0$ we get $v_{k-1}=1$. So assume that $f_0(v_{k-1})\neq v_{k-1}$ and denote
	$$
	i_0:=\max\{i\in\{0,1,\ldots,k-2\}: f_i(v_{k-1})\neq v_{k-1}\}.
	$$
	Since $f_{k-1}(v_{k-1})=v_{k-1}$, we obtain that $f_{i_0+1}(v_{k-1})=v_{k-1}$.
	 Since for each number $i=1,\ldots,k$ we have $f_i=(f_{i-1},f_{i-1}^2)$ or $f_i=(f_{i-1}^2,f_{i-1})$,
	 we must have 	$f_{i_0+1}=(f_{i_0},f_{i_0}^2)$ and therefore the point $v_{k-1}$ is $2$-periodic for $f_{i_0}$. Next observe that $v_{k-1}\notin(c,v_{i_0})$, since otherwise the map $f_{i_0+1}$ would have a fixed point inside the interval $(u_{i_0+1},v_{i_0+1})$, which is impossible due to the topological expanding condition for $f_{i_0+1}$. So $v_{k-1}=v_{i_0}$ and as a result we obtain 
	$$
	\{v_{k-1},c\}=\Orb_{f_{i_0}}(v_{k-1})\subset\Orb_{f_0}(v_{k-1}).
	$$
	Similarly we prove that $f_{k-1}(u_{k-1})=u_{k-1}$ implies $c\in\Orb_{f_0}(u_{k-1})$ or $u_{k-1}=0$.

	Recall that for every number $k=1,2,\ldots,n$ and point $x\in[u_k,v_k]$ we have inclusion $\Orb_{f_{k}}(x)\subset\Orb_{f_0}(x)$. Therefore, from the above considerations it follows that the orbit $\Orb_{f_0}(x)$ of any point $x\in(0,1)$ intersects the interval $(u_n,v_n)$. Indeed, $J_{f_n}\subset\{0,1\}$.
	
\textbf{(\ref{thm:renorm_comp:2}):} In this case the map $f_n=(f_{n-1}^2,f_{n-1}^2)$ is defined on the renormalization interval $[u_n,v_n]=[f_{n-1}(u_{n-1}),f_{n-1}(v_{n-1})]\subset[u_{n-1},v_{n-1}]$ and by definition for each number $i\in\N$ we have
	\begin{equation}\label{eq:(2,2)_renorm}
		f_{n-1}^{2i}\left([u_n,v_n]\right)\subset[u_n,v_n].
	\end{equation}	
	We will show that $f_{n-1}$ has a $2$-periodic orbit $O=\{P_L,P_R\}\subset [u_{n-1},v_{n-1}]$ such that
	$$
	P_L\leq u_n<c<v_n\leq P_R.
	$$
	It is clear that neither $u_n=0$ nor $v_n=1$, since $f_n=(f_{n-1}^2,f_{n-1}^2)$.		
	At first observe that there is a unique point $q\in[u_{n-1},c)\cap f_{n-1}^{-1}(v_n)$ and a unique $p\in(c,v_{n-1}]\cap f_{n-1}^{-1}(u_n)$. Suppose that $q>u_n$. Then
	$$
	f_{n-1}((u_n,q))=(f_{n-1}(u_n),v_n)\subset(c,v_n).
	$$
	So the condition~\eqref{eq:(2,2)_renorm} implies that for each $i\in\N$ we have
	$$
	f^i_{n-1}((u_n,q))\subset[u_n,v_n]
	$$
	and consequently $c\notin f^i_{n-1}((u_n,q))$, because $[u_n,v_n]
		\subsetneq [0,1]$. This contradicts the topological expanding condition for $f_{n-1}$ and  therefore $q\leq u_n$. Similarly we prove that $p\geq v_n$.

	Note that if $q=u_n$, then $f_{n-1}^2([u_n,c))=f_{n-1}([v_n,1))\subset[u_n,v_n]$ and so the condition~\eqref{eq:(2,2)_renorm} together with the inequality $p\geq v_n$ lead to $p=v_n$, and hence $O=\{P_L,P_R\}=\{u_n,v_n\}$. Next, let us consider the case $q<u_n$. Note that
	$$
	f_{n-1}\left([q,u_n]\right)=\left[v_n,f_{n-1}(u_n)\right]\subset(c,v_{n-1}]
	$$
	and that $f_{n-1}^2(v_n)\leq v_n=f_{n-1}(q)$ which gives
	$$
	\quad f_{n-1}(v_n)\leq q<u_n\leq f_{n-1}^{2}(u_n).
	$$
	This implies that
	$$
	[q,u_n]\subset f_{n-1}^2([q,u_n]),
	$$
	so there is a $2$-periodic orbit $O=\{P_L, P_R\}\subset[q,u_n]\cup[v_n,p]$ for the map $f_{n-1}$.

	From the first part of the proof it follows that
	$$
	f_{n-1}=\left(f_0^{l_{n-1}},f_0^{r_{n-1}}\right)\quad\text{and}\quad f_{n}=\left(f_{n-1}^2,f_{n-1}^2\right)=\left(f_0^{l_{n-1}+r_{n-1}},f_0^{l_{n-1}+r_{n-1}}\right)
	$$
	for some numbers $l_{n-1}$, $r_{n-1}\in\N$. Let us denote $l:=l_{n-1}+r_{n-1}$. Then
	$$
	P_L=f_{n-1}^2(P_L)=f_0^l(P_L)=f_n(P_L).
	$$
	Moreover $\Orb_{f_0}(P_L)\cap(u_n,v_n)=\emptyset$, because otherwise the map $f_n$ has a fixed point in $(u_n,v_n)$, which is impossible due to the topological expanding condition for $f_n$. So we have shown that the inclusion $\Orb_{f_0}(P_L)\subset J_{f_n}$ holds.
	
	It remains to show that $R_{J_{f_n}}f=f_n$, i.e.
	$$
	N\left( (e_-,c)\right)=N\left( (c,e_+)\right)=l=l_{n-1}+r_{n-1},
	$$
	where
	$$
	e_-:=\sup\{x\in J_{f_n}, x<c\}\quad\text{and}\quad e_+:=\inf\{x\in J_{f_n}, x>c\}
	$$
	(cf. Theorem~\ref{thm:A}). Firstly let us suppose that $P_L<e_-$. Then we get
	$$
	f_{n-1}\left( (P_L,e_-]\right)=\left(P_R,f_{n-1}(e_-)\right]\subset\left(P_R,v_{n-1}\right)
	$$
	and
	$$
	f_{n-1}\left( \left(P_R,f_{n-1}(e_-)\right]\right)=\left(P_L,f_{n-1}^2(e_-)\right]\subset\left(P_L,v_{n}\right).
	$$
	Therefore, by the definition of point $e_-$ we obtain
	$$
	f_{n-1}^2\left( (P_L,e_-]\right)\subset(P_L,e_-],
	$$
	which contradicts the topological expanding condition for $f_{n-1}$. So $P_L=e_-$ and similarly we prove that $P_R=e_+$. 
	As a consequence we get $J_{f_n}\cap(P_L,P_R)=\emptyset$, in particular
	\begin{equation}\label{eq:orbP_L}
	\Orb_{f_0}(P_L)\cap(P_L,P_R)=\emptyset.
	\end{equation}

	Let us denote $m:=|\Orb_{f_0}(P_L)|$, so the point $P_L$ is $m$-periodic for the map $f_0$. Observe that $N((P_L,c))\leq m$, since otherwise
	$$
	f_0^m\left((P_L,c)\right)=\left(P_L,f_0^{m-1}(1)\right)
	\subset (P_L,c)
	$$
	which is impossible due to the topological expanding condition for $f_0$. 

	Now suppose that $k:=N((P_L,c))<m$. From the condition~\eqref{eq:orbP_L} it follows that
	$$
	f_0^k\left((P_L,c)\right) = \left(f_0^k(P_L),f_0^{k-1}(1)\right)\supset[P_L,c].
	$$
	So there is a point $x\in(P_L,c)$ such that $f_0^k(x)=P_L$ and $c\notin\Orb_{f_0}(x)$. The equality $P_L=e_-$ implies that $f_0^i(x)\in(u_n,v_n)$ for some number $i\in\N_0$. But then for every number $j\in\N_0$ we have $f_n^j(f_0^i(x))\in(u_n,v_n)$, which contradicts the condition~\eqref{eq:orbP_L}. Therefore $N((P_L,c))=m$. Similarly we get the equality $N((c,P_R))=m$.

	Since $f_0^l(P_L)=P_L$ and $m$ is a period of $P_L$ we know that $m$ divides $l$. We claim that in fact $m=l$. For this purpose, suppose that $m<l=l_{n-1}+r_{n-1}$ and recall that
	$$
	f_{0}^{l_{n-1}}(P_L)=f_{n-1}(P_L)=P_R,\quad f_{0}^{r_{n-1}}(P_R)=f_{n-1}(P_R)=P_L.
	$$
	Let us denote
	$$
	i_0=\min\{i\in\N:f_0^i(P_L)=P_R\}\quad\text{and}\quad j_0=\min\{j\in\N:f_0^j(P_R)=P_L\}.
	$$
	So $m=i_0+j_0$. Furthermore, there are numbers $a$, $b\in\N_0$ such that $l_{n-1}=a\cdot m+i_0$ and $r_{n-1}=b\cdot m+j_0$. Observe that at least one of the numbers $a$ and $b$ is not equal to $0$, because $m<l$. Let us consider the case $a\neq0$ (the case $b\neq0$ is dealt similarly). Since $N((P_L,c))=N((c,P_R))=m$, we obtain
	$$
	f_0^m\left((P_L,c)\right)=\left(P_L,f_0^{m-1}(1)\right)\quad\text{and}\quad f_0^m\left((c,P_R)\right)=\left(f_0^{m-1}(0),P_R\right).
	$$
	Note that $f_0^{m-1}(1)\leq P_R$, because otherwise there is a point $x\in(P_L,c)$ such that $f_0^m(x)=P_R$ and $c\notin\Orb_{f_0}(x)$, which, as previously, leads to the contradiction with the condition~\eqref{eq:orbP_L}. Similarly $f_0^{m-1}(0)\geq P_L$. 
	Hence $f_0^m(y)\in(P_L,P_R)$ for every point $y\in(P_L,c)\cup(c,P_R)$. 
	
	Observe that
	$$
	f_0^m\left((P_L,c) \right)=\left(P_L,f_0^m(c_-)\right)\ni c\quad\text{and}\quad f_0^m\left((c,P_R) \right)=\left(f_0^m(c_+),P_R\right)\ni c.
	$$
	So in particular $c\in f_0^{a\cdot m}((P_L,c))$. But the map $f_0^{l_{n-1}}=f_0^{a\cdot m+i_0}$ is continuous on the interval $(P_L,c)\subset[u_{n-1},c)$ and hence $f_0^{i_0}(c_+)=f_0^{i_0}(c_-)$. On the other hand $f_0^m(c_+)\neq f^m(c_-)$, so for some number $s<m-i_0$ we have
	$$
	c=f_0^{i_0+s}(c_+)=f_0^{i_0+s}(c_-).
	$$
	Next, let us consider the interval $(c,P_R)$. Since $k:=i_0+s<m=N((c,P_R))$, we obtain
	$$
	f_0^k\left((c,P_R)\right)=\left(c,f_0^k(P_R)\right). 
	$$
	Observe that the interval $(P_L,P_R)$ cannot contain any preimage $y$ of the point $P_R$. Namely, either $y\neq c$, so $f^m(y)\in (P_L,P_R)$
		showing that $\Orb_{f_0}(P_L)\cap(P_L,P_R)\neq \emptyset$ or $y=c$ and then $P_R=v_n$ showing that $m\geq l$. In both cases, it is a contradiction.
	 Then the condition~\eqref{eq:orbP_L} implies that $f_0^k(P_R)=P_R$, which again contradicts the definition of $m$.
	Indeed $l=m=N\left( (e_-,c)\right)=N\left( (c,e_+)\right)$.

	\eqref{thm:renorm_comp:3}: Let us denote
	$$
	j:=\max\{i\in\{1,2,\ldots,n\}:f_i=(f_{i-1}^2,f_{i-1}^2)\}.
	$$
	Then, by the assumption we have $j<n$. Moreover it follows from \eqref{thm:renorm_comp:2} that the set $J_{f_j}$ contains a periodic orbit and $R_{J_{f_j}}f=f_j$. Since for each $k=j+1,\ldots,n$ the renormalization $f_k$ is trivial, we have 
	$$
	[u_k,v_k]=[f_{k-1}(u_{k-1}),v_{k-1}]\quad\text{or}\quad[u_k,v_k]=[u_{k-1},f_{k-1}(v_{k-1})].
	$$
	In particular, we conclude from it that
	$$
	[f_{j-1}(u_{j-1}),f_{j-1}(v_{j-1})]=[u_j,v_j]\supset[u_{j+1},v_{j+1}]\supset\ldots\supset[u_n,v_n]
	$$
	and as a consequence $J_{f_j}\subset J_{f_n}$. We claim that in fact $J_{f_j}=J_{f_n}$. Indeed, repeating the reasoning from the proof of \eqref{thm:renorm_comp:1}, we obtain
	$$
	\Orb_{f_j}(x)\cap(u_n,v_n)\neq\emptyset
	$$
	for every point $x\in(u_j,v_j)$ and
 	\begin{align*}
	\Orb_{f_j}(u_j)\cap(u_n,v_n)=\emptyset&\iff u_j=0,\\
	\Orb_{f_j}(v_j)\cap(u_n,v_n)=\emptyset&\iff v_j=1.
	\end{align*}
	On the other hand, the following implications hold
	\begin{align*}
	u_j=0&\implies f_{j-1}(u_{j-1})=0\implies u_{j-1}=0,\\
	v_j=1&\implies f_{j-1}(v_{j-1})=1\implies v_{j-1}=1.
	\end{align*}
	But the existence of the renormalization $f_j=(f_{j-1}^2,f_{j-1}^2)$ implies that $u_{j-1}\neq0$ and $v_{j-1}\neq1$. Therefore the orbit $\Orb_{f_j}(x)\subset\Orb_{f_0}(x)$ of any point $x\in[u_j,v_j]$ intersects the interval $(u_n,v_n)$. This proves the equality $J_{f_j}=J_{f_n}$. Finally, note that $R_{J_{f_n}}f=R_{J_{f_j}}f=f_j\neq f_n$.
\end{proof}

Theorem~\ref{thm:renorm_comp} is based on the assumption that the renormalization $g$ can be factorized into renormalizations of the form~\eqref{eq:basic_renorm}, that we will call \textit{basic renormalizations}.

\begin{cor}\label{cor:basic_renorm}
Let $f$ be an expanding Lorenz map and $g=(f^l,f^r)$ be a renormalization of $f$ that can be factorized into basic renormalizations. Then the set $J_g$ is nonempty and defines the renormalization $g$, i.e. $R_{J_g}f=g$, if and only if $l= r$.
\end{cor}

The above corollary suggests that the case $g=(f^l,f^r)$ with $l\neq r$ should be excluded in the definition of renormalization in \cite{Cui} and \cite{Ding}. Note, however, that our result is valid only for renormalizations $g=(f^l,f^r)$ that can be factorized into basic renormalizations. This condition is not always satisfied in the case of expanding Lorenz maps $f$ that do not have a constant slope (see the main example in \cite{ChOp}).

\begin{rem}\label{rem:basic_renorm}
By the results of \cite[Subsection~6.1.2]{GS} (see also \cite{Glen}) the kneading invariant of an expanding Lorenz map with a constant slope is finitely ($0\leq n<\infty$) renormalizable with each of the $n$ renormalizations being either trivial or by words $(w_+,w_-)=(10,01)$. However, it would be interesting to find an argument independent of the kneading theory, that in the case of maps with a constant slope any renormalization can be factorized into basic renormalizations.
\end{rem}

\section*{Acknowledgements}
The authors are grateful to referee of this paper and his numerous suggestions. His comments on kneading invariant and its renormalizations gave inspiration to writing Section~\ref{sec:basic_renormalizations} and algorithm presented there.

P. Oprocha was supported by National Science Centre, Poland (NCN), grant no. 2019/35/B/ST1/02239.

Publication co-financed by the Krakow University of Economics as part of the “Conference Activity Support program -- WAK-2024”.


\begin{thebibliography}{99}
\bibitem{RB} R. Alcaraz Barrera, \textit{Topological and ergodic properties of symmetric sub-shifts}. Discrete Contin. Dyn. Syst. \textbf{34(11)} (2014), 4459--4486.
\bibitem{BBK}	R. Alcaraz Barrera, S. Baker, D. Kong, \textit{Entropy, topological transitivity, and dimensional properties of unique $q$-expansions.} Trans. Amer. Math. Soc. \textbf{371(5)} (2019), 3209--3258.
\bibitem{ALMT} Ll. Alsed\`a, J. Llibre, M. Misiurewicz, Ch. Tresser, \textit{Periods and entropy for Lorenz-like maps}. Ann. Inst. Fourier (Grenoble) \textbf{39} (1989), 929--952. 
\bibitem{ABS1}  V. S.Afra\v{\i}movich, V. V. Bykov, L. P. Shil`nikov, \textit{On attracting structurally unstable limit sets of Lorenz attractor type}. (in Russian) Trudy Moskov. Mat. Obshch. \textbf{44} (1982), 150--212.

\bibitem{Barnsley} M.~Barnsley, W.~Steiner, A.~Vince, \textit{Critical itineraries of maps with constant slope and one discontinuity}, Math. Proc. Cambridge Philos. Soc., \textbf{157} (2014), 547--565.

	
\bibitem{Bart} P.~Bart\l{}omiejczyk, F.~Llovera Trujillo, J.~Signerska-Rynkowska, \textit{Spike patterns and chaos in a map-based neuron model}, Int. J. Appl. Math. Comput. Sci., \textbf{33} (2023), no. 3, 395--408.


	\bibitem{Bru1} H.~Bruin, C.~Carminati, C.~Kalle,
	\textit{Matching for generalised $\beta$-transformations}, 
	Indag. Math.~\textbf{28} (2017), 55--73.
	
	\bibitem{Bru2} H.~Bruin, C.~Carminati, S.~Marmi, A.~Profeti,
	\textit{Matching in a family of piecewise affine maps}, 
	Nonlinearity~\textbf{32} (2019), 172.
	
	
	\bibitem{ChOp} \L{}.~Cholewa, P.~Oprocha, \textit{On $\alpha$-Limit Sets in Lorenz Maps}, 
	Entropy, \textbf{23(9)} (2021), 1153.
	
	
	\bibitem{Clark} L. Clark, \textit{Maps with holes}, Ph. D. thesis, University of Manchester, 2016.
	
	\bibitem{Clark2} L. Clark, \textit{The $\beta$-transformation with a hole}. Discrete Contin. Dyn. Syst. \textbf{36(3)} (2016), 1249--1269.
	
	\bibitem{CT} V. Climenhaga, D. Thompson, \textit{Unique equilibrium states for flows and homeomorphisms with non-uniform structure}. Adv. Math. \textbf{303} (2016), 745--799.
	
	\bibitem{Cui} H.~Cui, Y.~Ding, \textit{Renormalization and 
		conjugacy of piecewise linear Lorenz maps}, 
	Adv.\ Math.~\textbf{271} (2015), 235--272.
	
	\bibitem{KHK} K. Dajani, Y. Hartono, C. Kraaikamp, \textit{Mixing properties of $(\alpha,\beta)$-expansions}. Ergodic Theory Dynam. Systems \textbf{29} (2009), no. 4, 1119--1140. 

	\bibitem{Ding} Y.~Ding, \textit{Renormalization and 
		$\alpha$-limit set for expanding Lorenz maps}. Discrete 
	Contin.\ Dyn.\ Syst.~\textbf{29} (2011), 979--999.
	
	
	\bibitem{DingSun} Y.~Ding, Y.~Sun, \textit{Complete invariants and parametrization of expansive Lorenz maps}, arXiv:~2103.16979.
	
	
	
	\bibitem{GellerMis} W.~Geller, M.~Misiurewicz, \textit{Farey-Lorenz Permutations for Interval Maps}, Int. J. Bifurcation Chaos, \textbf{28(02)} (2018): 1850021.

\bibitem{Glen}  P.~Glendinning, \textit{Topological 
	conjugation of Lorenz maps by $\beta$-transformations}, 
Math.\ Proc.\ Cambridge Philos.\ Soc.~\textbf{107} 
(1990), 401--413.

\bibitem{GH} P. Glendinning, T. Hall, \textit{Zeros of the kneading invariant and topological entropy for Lorenz maps}, Nonlinearity, \textbf{9} (1996), 999--1014.

\bibitem{GS} P. Glendinning, C. Sparrow, \textit{Prime and renormalizable kneading invariants and the dynamics of expanding Lorenz maps}, Physica D, \textbf{62} (1993), 22--50.

\bibitem{Gu1} J. Guckenheimer, \textit{A strange, strange attractor}, in: J. E. Marsden and M. McCracken (eds.), The Hopf Bifurcation Theorem and its Applications, Springer, 1976, pp. 368--381.

\bibitem{Hof1} F. Hofbauer, \textit{Periodic points for piecewise monotonic transformations}. Ergodic Theory Dynam. Systems \textbf{5} (1985),  237--256.

\bibitem{Hof2} F.~Hofbauer, \textit{Piecewise invertible 
	dynamical systems}, Probab.\ Theory Relat.\ 
Fields.~\textbf{72} (1986), 359--386.

	\bibitem{HubSpar} J. H. Hubbard, C. Sparrow, \textit{The classification of topologically expansive Lorenz maps}, Comm. Pure Appl. Math. \textbf{43} (1990), 431--443.
		
	\bibitem{Kam} A.~Kameyama, \textit{Topological transitivity and strong transitivity}, Acta Math. Univ. Comenianae (N.S.) 71.2 (2002), 139--145.
	
\bibitem{MSP1} G. Keller, M.  St. Pierre, \textit{Topological and measurable dynamics of Lorenz maps}. Ergodic theory, analysis, and efficient simulation of dynamical systems, 333--361, Springer, Berlin, 2001. 

\bibitem{SL} S. Luzzatto, I. Melbourne, F. Paccaut, \textit{The Lorenz attractor is mixing}. Comm. Math. Phys. \textbf{260(2)} (2005), 393--401.

\bibitem{Mar} M. Martens. \textit{The periodic points of renormalization}. Ann. of Math. (2), \textbf{147} (1998), 543--584.

\bibitem{MW94} M. Martens, B. Winckler, \textit{On the hyperbolicity of Lorenz renormalization}. Comm. Math. Phys. \textbf{325} (2014), no. 1, 185--257.

\bibitem{MW} M. Martens, B. Winckler, \textit{Physical measures for infinitely renormalizable Lorenz maps}. Ergodic Theory Dynam. Systems \textbf{38(2)} (2018), 717--738. 

\bibitem{Schme} J. Schmeling. \textit{Symbolic dynamics for $\beta$-shifts and self-normal numbers}. Ergodic Theory Dynam. Systems \textbf{17(3)} (1997), 675--694. 

\bibitem{MSP2} M. St. Pierre, \textit{Topological and measurable dynamics of Lorenz maps}. Dissertationes Math. \textbf{382} (1999), 1--134.

\bibitem{OPR} P. Oprocha, P. Potorski, P. Raith, \textit{Mixing properties in expanding Lorenz maps}. Adv. Math. \textbf{343} (2019), 712--755

\bibitem{Palmer} M. R. Palmer, \textit{On the Classification of Measure Preserving Transformations of Lebesgue Spaces}, Ph. D. thesis, University of Warwick, 1979.

\bibitem{Rai2} P.~Raith, \textit{Continuity of the Hausdorff dimension for piecewise monotonic maps}. Israel J. Math. \textbf{80} (1992),  97--133.

\bibitem{Tho} D. Thompson, \textit{Irregular sets, the $\beta$-transformation and the almost specification property}. Trans. Amer. Math. Soc. \textbf{364(10)} (2012), 5395--5414. 

\bibitem{Wi1} R. F. Williams, \textit{The structure of Lorenz attractors}. Inst. Hautes \`Etudes Sci. Publ. Math. No. \textbf{50}, (1979), 73--99

\bibitem{Winc} B. Winckler, \textit{Renormalization of Lorenz Maps}. PhD Thesis KTH, Stockholm, Sweden, 2011.
\end{thebibliography}
\end{document}